\documentclass[10pt]{article}
\usepackage{amsmath,amssymb,amsthm,amscd,esint}
\usepackage{appendix}

\newtheorem{theorem}{Theorem}[section]
\newtheorem{lemma}[theorem]{Lemma}

\newtheorem{corollary}[theorem]{Corollary}
\newtheorem{proposition}[theorem]{Proposition}

\theoremstyle{definition}
\newtheorem{definition}[theorem]{Definition}
\newtheorem{example}[theorem]{Example}

\theoremstyle{remark}
\newtheorem{remark}[theorem]{Remark}

\numberwithin{equation}{section}
\numberwithin{theorem}{section}

\DeclareMathOperator{\vol}{vol}

\DeclareMathOperator{\tr}{tr} 
\DeclareMathOperator{\diam}{diam}

\DeclareMathOperator{\Psh}{Psh}
\DeclareMathOperator{\Ca}{Cap}

\begin{document}

\title{Complex Monge-Amp\`ere equation in Orlicz space and Diameter Bound}

\author{Lei Zhang\thanks{Supported by Postdoctoral Fellowship Program GZC20240867. Email: leizhang92@mail.tsinghua.edu.cn}\\
Yau Mathematical Sciences Center, Tsinghua University\\
Zhenlei Zhang\thanks{Supported partially by NSFC 11431009 and NSFC 11771301. Email: zhleigo@aliyun.com}\\
School of Mathematical Sciences, Capital Normal University}
\date{}

\maketitle



\begin{abstract}
In this paper, we establish diameter bounds for compact K\"ahler manifolds equipped with K\"ahler metrics $\omega$, assuming the associated measure lies in a specific Orlicz space and satisfies an integrability condition. Firstly, we prove a priori estimates for solutions of the complex Monge-Amp\`ere equation in Orlicz spaces, encompassing $L^{\infty}$ and stability estimates. This is achieved by employing Ko{\l}odziej's approach \cite{Ko98} and the argument of Guo-Phong-Tong-Wang \cite{GuPhToWa21}, respectively. Secondly, building on the work of Guo-Phong-Song-Sturm \cite{GuPhSoSt24-1}, we derive the uniform (local/global) estimates of the Green's function and its gradient for the associated K\"ahler metric $\omega$.
\end{abstract}

\tableofcontents



\section{Introduction}
In a celetrated article \cite{Ya78}, S. T. Yau solved the Calabi conjecture by studying complex Monge-Amp\`ere equations (CMA for short) on a compact K\"ahler manifold. The $L^{\infty}$-estimates are of particular importance and played a prominent role in Yau's proof on the Calabi conjecture. Since then, the CMA has attracted much attentions in complex geometry. The study of the CMA on compact K\"ahler manifolds has a long history and many spectacular results have appeared over the years. For a comprehensive overview of these developments, readers are referred to the surveys by Phong-Song-Sturm \cite{PhSoSt12} and Guo-Phong \cite{GuPh23}.

Major developments in the theory of CMA occurred in the last decades. The pioneering work of E. Bedford and B. A. Taylor \cite{BeTa76,BeTa82} has been deeply analysed by Ko{\l}odziej \cite{Ko98,Ko03,Ko05} which gives a pluripotential approach proof of $L^{\infty}$-estimates for CMA when the right hand is in $L^p$ for $p>1.$ This improvement in the range of $p$ is not merely technical but bears profound geometric significance. The failure of $L^{\infty}$-estimates at $p=1$ underscores the necessity of stability conditions in the K\"ahler-Ricci flow \cite{PhSoSt07}, whereas the regime $1 < p \leq n$ is essential for a broad spectrum of geometric problems. Actually, an extension of Ko{\l}odziej's methods to the more general case of degenerating background metrics is necessary. This framework was independently revisited by Demailly-Pail \cite{DePa10} and Eyssidieux-Guedj-Zeriahi \cite{EyGuZe09} in the context of singular canonical metrics. It was subsequently extended to general big cohomology classes by Boucksom-Eyssidieux-Guedj-Zeriahi \cite{BEGZ10}, and later to the nef classes by Fu-Guo-Song \cite{FGS20}. Utilizing the Alexandrov-Bakelman-Pucci (ABP) maximum principle, Blocki established an alternative framework \cite{Blo05,Blo11} in which $L^{\infty}$-estimates for CMA when the right hand is in $L^p$ for $p>2.$ His methods have proven remarkably powerful and have since been successfully applied to numerous problems, such as subsolutions, equations with gradient terms, and the constant scalar curvature problem.

Very recently, significant advancements have been made in proving Ko{\l}odziej's estimates through alternative methodologies. In \cite{WaWaZh21}, Wang-Wang-Zhou obtained an elegant PDE proof of sharp $L^{\infty}$ estimates for the CMA for domains in $\mathbb{C}^n$. A purely PDE-based approach for establishing sharp $L^{\infty}$ estimates for the CMA on compact K\"ahler manifold was developed by Guo-Phong-Tong \cite{GuPhTo23-2}, extending their earlier work to a broader class of fully non-linear equations \cite{GuPh24} while simultaneously accommodating degenerations in the background metric \cite{GuPhTo23-1,GuPhToWa24}. The PDE approach has seen substantial development in subsequent works, including further studies by Guo-Phong-Tong \cite{GuPhTo23-2,GuPhTo23-1}, their collaborative work with Wang \cite{GuPhToWa24,GuPhToWa21}, and their joint research with Sturm \cite{GuPhSt24} and Song \cite{GuPhSoSt24-1,GuPhSoSt23,GuPhSoSt24-2}. These developments parallel broader applications of PDE estimates across various geometrical contexts. These methods are especially well-suited to geometric applications and have led to promising results concerning diameter, volume and Green's function estimates \cite{GuPhSt24,GuPhToWa21,GuPhSoSt23,GuPhSoSt24-1,GuPhSoSt24-2,GuSo22,Vu24+}. This PDE's method provides an alternative to the traditional pluripotential theory techniques. Concurrently, an approach utilizing quasi-psh envelopes was established by Guedj-Lu-Zeriahi \cite{GuLuZe19} and further refined by Guedj-Lu \cite{GuLu21,GuLu22,GuLu23}. We note that similar estimates have also been obtained through pluripotential methods by Guedj-Guenancia-Zeriahi \cite{GuGuZe25} and through complex Sobolev space by Dinh-Ko{\l}odziej-Nguyen \cite{DiKoNg21} and Vu \cite{Vu24+}.

The stability estimates for CMA provide good control of the potentials of those metrics close to the singularity set. Informally, the stability estimate shows that $L^1$-closeness of right-hand sides in a fixed Orlicz space implies $L^\infty$-close of the potentials. The foundational work for the CMA was done by Ko{\l}odziej \cite{Ko03}. This was generalized to semipositive and big background metrics by Dinew-Zhang \cite{DiZh10} (see also \cite{GuKoZe08,GuZe12}). Based on the stability estimates for CMA, Ko{\l}odziej \cite{Ko08}, in particular proved the potentials for CMA (when the right hand is in $L^q$ for $q>1$) are H\"older continuous. Later, in \cite{DDGHKZ}, the authors combine the methods of Ko{\l}odziej and Demailly's regularization theorem \cite{De92,De94} to obtain the H\"older continuous with exponent $\alpha$ arbitrarily close to $\frac{2}{1+nq}$, where $q$ is the conjugate exponent of $p$. For the Dirichlet problem of the CMA, we refer to \cite{BeTa76,GuKoZe08,Ng18,Ng20}; for characterization of Monge-Amp\`ere measures with H\"older continuous potentials, see \cite{DiNgSi10,DiNg14}.

However, there exist counterexamples demonstrating that H\"older continuity can fail when the right-hand side is not in $L^p$ for any $p>1$ (see \cite[Example 1]{GuPhToWa21}). Notably, Li's recent work \cite{Li21} established a fundamental connection between the modulus of continuity of potentials and diameter bounds, with further advancements provided in \cite{GuPhToWa21,GuGuZe25}. For the CMA, this modulus is critically important as it directly governs key geometric properties of the associated K\"ahler metric, including its diameter \cite{FGS20,Li21,GuPhToWa21,GuPhSoSt24-1,GuPhSoSt23,GuGuZe25,Vu24+} and volume \cite{GuSo22, GuPhSoSt24-1,GuPhSoSt23}. Crucially, uniform bounds on the modulus are necessary to establish uniform diameter bounds, which are in turn essential for controlling Gromov-Hausdorff convergence. These results have found numerous applications, notably in the analytic minimal model program via the K\"ahler-Ricci flow \cite{TiZh06,SoTi07,SoTi12,SoTi17} and the continuity method \cite{LaTi16,LaTiZh17,FGS20,Zha24}, as well as in the study of degenerating Calabi-Yau metrics \cite{To09,To10,RoZh11}.

In this paper, we will establish a priori estimates which including $L^{\infty}$ and stability estimates for solutions to the CMA (\ref{MA: smooth case}) in Orlicz spaces, by employing Ko{\l}odziej's original approach \cite{Ko98} and the argument of Guo-Phong-Tong-Wang \cite{GuPhToWa21}, respectively. Building on these analytical results and Guo-Phong-Song-Sturm's discussion \cite{GuPhSoSt24-1}, we further derive geometric insights into the associated K\"ahler metrics, such as their diameter and volume.

Let's start with the general set-up. We adopt the terminology of \cite{GuPhSoSt24-1}. Let $X^n$ be a compact K\"ahler manifold of dimension $n\ge 2$ and $\omega_0$ be a fixed K\"ahler metric. Let $\theta$ be a $d$-closed smooth real (1,1)-form in a big class $[\alpha]$ such that
\begin{equation}\label{harmonic form bound}
\theta\le A\omega_0
\end{equation}
for some constant $A>0$. Consider the following CMA
\begin{equation}\label{MA: smooth case}
\frac{\left(\theta+\sqrt{-1}\partial\bar{\partial}\varphi_0\right)^n}{[\alpha^n]}=F\cdot\frac{\omega_0^n}{[\omega_0^n]},~~~\sup\varphi_0=0,
\end{equation}
where $F$ is a positive measurable function in an Orlicz space $L^\Phi$ for some $N$-function $\Phi$. The measure also satisfies a compatibility condition
\begin{equation}\label{prob measure}
\fint_X F\omega_0^n=1,
\end{equation}
where $\fint$ denotes the average integration w.r.t. the measure $\omega_0^n$.

See Section \ref{Orlicz-space} for a brief introduction to the Orlicz space. For any $f\in L^\Phi(\omega_0)$ we denote by
\begin{equation}
\|f\|_{L^\Phi}=\inf \left\{b>0\,\bigg|\,\frac{1}{[\omega_0^n]}\cdot\int_X\Phi(|f|/b)\omega_0^n\le 1\right\}
\end{equation}
the Luxemburg norm of $f$ with respect to the probability measure $\frac{\omega_0^n}{[\omega_0^n]}$.

\subsection{A priori estimates for CMA in Orlicz space}

The a priori analytical estimates of CMA in Orlicz space depend on integrability conditions of $\Phi$. Recall the "convex conjugate" of an $N$-function $\Phi$ (which it is called in \cite{KrRu} the {\it complementary} function of $\Phi$),
\begin{equation}
\Phi^*(s)=\sup_{t\ge 0}(ts-\Phi(t)),\hspace{0.3cm}\forall ~s\ge 0.
\end{equation}
Let us introduce another function associated to $\Phi$,
\begin{equation}\label{N function-4}
h(t)=h_\Phi(t)=(\Phi^*)^{-1}(e^t).
\end{equation}
Now we introduce the first integrability condition on $\Phi$, namely
\begin{equation}\label{integrability 0}
\mathbf{I}=\int_0^\infty\frac{dt}{h^{1/n}(t)}<\infty.
\end{equation}
The condition, which is called Condition (A) in \cite[Chapter 2]{Ko98}, is crucial for the $L^\infty$ estimate of CMA (\ref{MA: smooth case}) in Orlicz spaces. We will take Ko{\l}odziej iteration argument for capacity \cite{Ko98} to prove the following $L^\infty$-estimate. Related discussions can be found in \cite[Theorem 2.1]{EyGuZe09}, \cite[Theorem 4.1]{BEGZ10}, and \cite[Proposition 1.1]{FGS20}. One may also use the auxilliary function of \cite{GuPhTo23-2} to give another proof (see also \cite{Liu24+,Qiao24+,ZwZy25+}).

\begin{theorem}\label{infinity estimate: 2}
Let $\Phi$ be an $N$-function satisfying {\rm(\ref{integrability 0})}. If $F\in L^\Phi$ and its Luxemburg norm
\begin{equation}\label{Orlicz measure}
\|F\|_{L^\Phi(\omega_0)}\le K,
\end{equation}
then, the solution $\varphi_0$ of CMA {\rm(\ref{MA: smooth case})} admits an a priori $L^\infty$-estimate
\begin{equation}
\|\varphi_0-\mathcal{V}_\theta\|\le C(\omega_0,\Phi;A,K,n)+C(\omega_0,A)\cdot\|F\|_{L^\Phi(\omega_0)}^{1/n}\cdot\mathbf{I},
\end{equation}
where $$\mathcal{V}_\theta(x)=\sup\{\psi(x)\,|\,\psi\in\Psh(\theta),\,\sup\psi=0\}.$$
\end{theorem}

We remark that the estimate is a formal result in general Orlicz space. Several specific examples are remarkable. An application of Theorem \ref{infinity estimate: 2} to $\Phi(t) = t^p/p$ (with $p > 1$) yields a general estimate, from which we directly recover Ko{\l}odziej's result for K\"ahler classes \cite[Theorem 2.4.2]{Ko98}, the estimate by Boucksom-Eyssidieux-Guedj-Zeriahi for big cohomology classes \cite[Theorem 4.1]{BEGZ10}, and the estimate for nef classes by Fu-Guo-Song \cite[Proposition 1.1]{FGS20} (where we choose $\theta=\chi+t\omega_0$ for some smooth closed $(1,1)$-form $\chi$ in nef class). Applying it to $\Phi(t)=t\cdot\log(1+t)^p$ for $p>n$, we establish the $L^{\infty}$-estimates of Guo-Phong-Tong \cite[Theorem 1]{GuPhTo23-2}. See the discussion in Section \ref{example section}.

The stability estimates of CMA relates to a second integrability condition
\begin{equation}\label{integrability 00}
\int_{h(\tau_0)}^\infty\frac{dt}{t^{1/n}\cdot (h^*)^{-1}(t)}<\infty,
\end{equation}
where $h^*$ is the complementary of $h$ and $\tau_0$ is any positive number. The stability estimates below are obtained by adapting the method of Guo-Phong-Tong-Wang \cite{GuPhToWa21}. Further details are provided in Section \ref{proof of stability of cma}.

\begin{theorem}\label{Stability theorem 1}
Suppose that the cohomology class $[\theta]$ is nef and big. Let $\varphi_0$ be a solution  to CMA {\rm(\ref{MA: smooth case})}. Assume further that $\Phi$ satisfies the integrability condition {\rm(\ref{integrability 00})}. Then, for any $\psi_\delta\in\Psh(\theta)$ satisfying
\begin{equation}
\sup\psi_\delta= 0
\end{equation}
and
\begin{equation}
\fint_X(\psi_\delta-\varphi_0)_+\omega_0^n\le \delta,
\end{equation}
we have
\begin{equation}\label{stability estimate}
\psi_\delta\le\varphi_0+C\cdot \hbar(\delta).
\end{equation}
for some $C=C(\omega_0,\Phi;A,n,K)$, where $f_+=\max\{f,0\}$, $\hbar(\delta)$ is defined in Definition {\rm\ref{Definition-hbar}}.
\end{theorem}

We remark that the formal estimate generalizes Ko{\l}odziej's stability as well as the continuity of solutions to CMA. For CMA with right hand side in $L^p$ (or $\Phi(t)=\frac{t^p}{p}$) for $p>1$, the stability is proved in Ko{\l}odziej \cite[Theorem 4.1]{Ko03} and \cite[Theorem 2.5.2]{Ko98}; for right hand side in $L^1(\log L)^p$ (or $\Phi(t)=t\cdot\log(1+t)^p$) for $p>n$, it is proved by Guo-Phong-Tong \cite[Theorem 1]{GuPhTo23-1} and \cite[Theorem 1]{GuPhTo23-2}. Applying it to $\Phi(t)=t\cdot\log(1+t)^n\cdot(\log\circ\log(1+t))^p$ for $p>n$, we establish the stability estimates of Guedj-Guenancia-Zeriahi \cite[Theorem 1.6]{GuGuZe25} (see also \cite[Theorem 1.3]{Liu24+} for related discussion). See the discussion in Section \ref{example section}.

Such a stability estimate played an important role in proving the H\"older continuity for solutions of CMA with right hand side in $L^p$ for $p>1$ (cf. \cite[Theorem 2.1]{Ko08} and \cite[Theorem A]{DDGHKZ}) and the modulus of continuity for solutions of CMA with right hand side in $L^1(\log L)^p$ for $p>n$ (See \cite[Theorem 1]{GuPhToWa21} \cite[Theorem 1.6]{GuGuZe25} and the references therein). The proof of \cite[Theorem A]{DDGHKZ} uses Demailly's regularization theorem \cite[Theorem 1.1]{De94} as in Berman-Demailly \cite{BeDe12}. The Kiselman minimum principle \cite{Ki78} coupled with Demailly's method of attenuating singularities (the Kiseman-Legendre transform) from \cite{De94} can be allowed us to remove symmetry/curvature constraints.

\begin{remark}
More properties of the solution $\varphi_0$ can be deduced. For example, one can estimate the modulus of continuity for solutions $\varphi_0$ to the CMA (\ref{MA: smooth case}) with right-hand side in an Orlicz space which satisfies integrability condition (\ref{integrability 00}) by applying the strategy of \cite{BeDe12,DDGHKZ}. Then, as in \cite[Theorem 4.1]{Li21}, the finiteness of the diameter is a consequence of the modulus of continuity for this solution $\varphi_0$ (see \cite[Theorem 2]{GuPhToWa21} for related discussion). A full discussion of this topic can be found in a separate paper \cite{ZhZh25-1}.
\end{remark}

\subsection{Geometric estimates of K\"ahler metrics}

Let $(X,\omega_0)$ be a K\"ahler manifold. Let $\varphi_0$ be a solution  to CMA {\rm(\ref{MA: smooth case})}. Assume that $\omega=\theta+\sqrt{-1}\partial\bar{\partial}\varphi_0$ is a K\"ahler metric.

Considerable work has been devoted in recent decades to establishing diameter upper bounds for K\"ahler manifolds, especially in the context of degenerate K\"ahler-Einstein families. This has resulted in a substantial body of literature on the subject (see, e.g., \cite{Pau01,To09,To10,RoZh11,FGS20,Li21,GuPhToWa21,GuPhSoSt24-1,GuGuZe25,GuTo25} and references therein). We will obtain a diameter bounds and volume noncollapsing estimates for the K\"ahler metrics associated with the CMA in Orlicz space with uniformly bounded potentials by adapting the methods of \cite[Section 5]{GuPhSoSt24-1} and \cite[Section 3]{GuPhSt24}.

We begin by establishing global integral bounds for the Green function and its gradient, associated with the K\"ahler metric $\omega$ arising from the CMA in an Orlicz space, under additional restrictions on the Monge-Amp\`ere measure (See Proposition \ref{L1+epsilon-Green}, Lemma \ref{Gradient-Green} in Section \ref{Refined integral estimate of Green function}). As a consequence, we obtain an upper bound for the diameter of the manifold with respect to $\omega$. A precise explanation is provided in Theorem \ref{main theorem} (Section \ref{Refined integral estimate of Green function}).

We then establish a local integral bound for the Green function associated with the metric $\omega$ on any geodesic ball $B_r(x_0)$ of radius $r$ centered at a point $x_0 \in X$, where the ball is defined with respect to $\omega$. The detailed proof is provided in Section \ref{local volume of geodesic balls}. This yields an improved volume lower bound for geodesic balls, refining the earlier results. 
The detailed discussion is provided in Section \ref{SG-N-function-diam-bound}, Section \ref{Polynomial growth N-functions for diamter-volume}.

Define the iterated logarithmic functions $\{g_k\}_{k=0}^{\infty}$ recursively via
$$g_0(t)=t,\,g_1(t)=\log (1+t),\,\cdots,g_{k+1}(t)=\log(1+g_k(t)),\cdots.$$
Assume that
$$\Phi\ge g_0(t)\cdot g_1(t)^n\cdot g_2(t)^{2n}\cdots g_{k-1}(t)^{2n}g_k(t)^p$$
for some $k\ge 2$ and $p>2n$, whenever $t$ is sufficiently large.

\begin{theorem}\label{diameter-bound-main-result}
Assume that $F\in L^\Phi$ for some $N$-function $\Phi$ introduced above. Then its diameter of $\omega$ is uniformly bounded
\begin{equation}\label{diameter bound in Orlicz space}
\diam(X,\omega)\le D,
\end{equation}
for some $D$ depending on $\omega_0,\Phi,A,K,n$. The volume of metric ball satisfies the lower bound
\begin{equation}\label{volume bound in Orlicz space}
\frac{\vol(B_r(x))}{[\omega^n]}\ge v(r)
\end{equation}
where $v(r)$ is a positive function of $r$, depending on $\omega_0,\Phi,A,K,n$, so that
$$\lim_{r\rightarrow 0}v(r)=0.$$
\end{theorem}

We remark that the function $v(r)$ admits a concise expression; see Section \ref{SG-N-function-diam-bound}, Section \ref{Polynomial growth N-functions for diamter-volume} for details. In fact, for the specific cases
$$\Phi(t)=t\log^pt~(p>n),\qquad \Phi(t)=t^p/p~(p>1)$$
see equation (\ref{almost-sharp-volume-estimate-1}) in Section \ref{SG-N-function-diam-bound} and equation (\ref{almost-sharp-volume-estimate}) in Section \ref{Polynomial growth N-functions for diamter-volume}, respectively. In particular, our theorem extends prior diameter bounds from \cite[Theorem 1.2]{FGS20}, \cite[Theorem 2]{GuPhToWa21}, \cite[Theorem 1.1]{GuPhSoSt24-1}, \cite[Theorem B]{GuGuZe25}, \cite[Theorem B]{GuTo25}, \cite[Theorem 1.4]{Liu24+} and \cite[Theorem 1.1]{Vu24+}, while for volume estimates, it provides a generalization and refinement of results in \cite[Theorem 1.1]{GuPhSoSt24-1}, \cite[Theorem B]{GuTo25}, \cite[Corollary 7.8]{Liu24+}, \cite[Theorem 1.1]{Vu24+} and \cite[Theorem 1.1, 1.2]{ZwZy25+}. The derivation of the diameter bound and the volume non-collapsing estimates required only the $L^{\infty}$-estimates for the solution of the CMA (\ref{MA: smooth case}) with more restrictions on Monge-Amp\`ere measure. The method we employ stands in contrast to that of \cite[Theorem 4.1]{Li21} and \cite[Theorem 2]{GuPhToWa21}. Our approach relies crucially on a technique introduced in \cite{GuPhTo23-2,GuPhSoSt24-1}, which entails comparing the original equation with an auxiliary CMA.

\begin{remark}
In August 2023, the second author was invited to deliver a short course on the complex Monge-Amp\`ere equation in Orlicz spaces at Hunan University. Some of the results contained in this paper were presented during that course. Both authors would like to express their gratitude to Yashan Zhang for his warm hospitality and for providing an excellent research environment during their visits to Hunan University.
\end{remark}

In a forthcoming paper we shall discuss more detailed geometry structure of $\omega$, such as the modulus of continuity of distance function $d_{\omega}$ induced by $\omega=\theta+\sqrt{-1}\partial\bar{\partial}\varphi_0$, by introducing some stronger integrability conditions.

{\bf Acknowledgment}: We would like to thank Quang-Tuan Dang, Jiyuan Han, and Bi Zhou for their insightful comments, for bringing to our attention an error in Step 2 of the proof of Theorem \ref{Stability theorem 1}, and for suggesting a concise remedy. Their contributions have significantly enhanced the quality of this paper.



\section{Preliminary results}

\subsection{Orlicz space}\label{Orlicz-space}

There is a classical reference for Orlicz space \cite{KrRu}. We just recall some basic material that will be used in the paper.

Let $\Phi:[0,\infty)\rightarrow[0,\infty)$ be a continuous and convex function that satisfies
\begin{equation}\label{N function}
\Phi(0)=0,\hspace{0.3cm}\lim_{t\rightarrow 0}\frac{\Phi(t)}{t}=0,\hspace{0.3cm}\lim_{t\rightarrow\infty}\frac{\Phi(t)}{t}=\infty.
\end{equation}
Such function is usually called an {\it $N$-function} in Orlicz space theory.

\begin{definition}[Orlicz space]
Let $(X,\mu)$ be a measure space. Assume that $\mu$ is a $\sigma$ finite positive measure. The space
\begin{equation}
L^\Phi(\mu)=\left\{f\mbox{ measurable }\,\bigg|\,\int_X\Phi(a|f|)d\mu<\infty\mbox{ for some }a>0.\right\}
\end{equation}
with Luxemburg norm
\begin{equation}
\|f\|_{L^\Phi}=\inf\left\{b>0\,\bigg|\,\int_X\Phi(|f|/b)d\mu\le 1\right\}
\end{equation}
is called the Orlicz space associated to $\Phi$. It is a complete Banach space.
\end{definition}

For any continuous function $f:[0,\infty)\rightarrow[0,\infty)$ define the "convex conjugation", which it is called in \cite{KrRu} the {\it complementary} function of $f$,
\begin{equation}\label{N function-1}
f^*(s)=\sup\{st-f(t)\,|\,t\ge 0\}.
\end{equation}
It is obvious that $f^*$ is monotone increasing on $[0,\infty)$. The supremum is always achieved, for any $s\ge 0$, at some $t^*=t^*(s)$ if $\lim_{t\rightarrow\infty}\frac{f(t)}{t}=\infty$. Under the later assumption the complementary $f^*$ grows at least linearly, so it is proper. Moreover, there is a universal inequality, usually called {\it Young inequality},
\begin{equation}\label{N function-2}
s\cdot t\le f(s)+f^*(t),\hspace{0.3cm}\forall s,t\ge 0.
\end{equation}

Let $\Phi$ be an $N$-function, then $\Phi(t)/t$ is increasing in $t$. In particular,
$$c\cdot\Phi\left(\frac{t}{c}\right)\le\Phi(t),\quad\forall t\ge 0, c\ge 1.$$
It follows that if $\int_X\Phi(|f|)d\mu\le c$ for some $c\ge 1$, then $\|f\|_{L^\Phi}\le c$.

Let $\Phi^*$ be the complementary of an $N$-function $\Phi$. There is a universal inequality
\begin{equation}\label{N function-3}
t<\Phi^{-1}(t)\cdot(\Phi^*)^{-1}(t)\le 2t,\hspace{0.3cm}\forall t>0.
\end{equation}
Let us introduce another function associated to $\Phi$,
\begin{equation}\label{N function-4b}
h(t)=h_\Phi(t)=(\Phi^*)^{-1}(e^t).
\end{equation}
It has a trivial lower bound
\begin{equation}
h(t)\ge\frac{e^t}{\Phi^{-1}(e^t)}.
\end{equation}

Given any two $N$-function $\Phi_1,~\Phi_2$. By (\ref{N function-1}) and (\ref{N function-4}), we see that if $\Phi_1(t)<\Phi_2(t)$, then
\begin{equation}\label{phi1-phi2}
\Phi_1^*(s)>\Phi_2^*(s),~~~h_1(s)<h_2(s).
\end{equation}
It is known that $\Phi^*(t^2/2)=t^2/2$, so $\Phi(t)<\Phi^*(t)$ for large $t$ if $\Phi$ grows less than a polynomial of order $2$.

The composition of two $N$-functions is also an $N$-function. On the other hand, given $N$-functions $\Phi$ and $\Phi_1$ such that $\Phi>\Phi_1$, then
$$\Phi_2=\Phi\circ\Phi_1^{-1}$$
is also an $N$-function, see \cite[P10]{KrRu}, which satisfies $$\lim_{t\to+\infty}\frac{\Phi_2^*(t)}{\Phi^*(t)}=+\infty,~~~\lim_{t\to+\infty}\frac{(\Phi_2^*)^{-1}(t)}{(\Phi^*)^{-1}(t)}=0.$$

In Section \ref{example section} several typical examples of $N$-functions and their relation to stability of CMA (\ref{MA: smooth case}) are discussed in full detail.

\subsection{Integrability conditions for $N$-function}\label{integrability}

Let $\Phi$ be an $N$-function and $h(t)=(\Phi^*)^{-1}(e^t)$. The solvability of the CMA for a measure $F$ in Orlicz space $L^\Phi$ is related to an integrability condition
\begin{equation}\label{integrability condition: 1}
\int_0^\infty\frac{dt}{h^{1/n}(t)}<\infty.
\end{equation}
The condition (\ref{integrability condition: 1}) implies that
\begin{equation}\label{limit 2}
\lim_{t\rightarrow\infty}\frac{t}{h^{1/n}(t)}=0.
\end{equation}
It is trivial that $h^*$ grows at least linearly. Let $\tau_0=h(1)$. By definition we have that $h^*(s)>0$ and increase strictly whenever $s>\tau_0$.

We now proceed to examine the following integrability condition
\begin{equation}\label{integrability condition: 2}
\int_{h(\tau_0)}^\infty\frac{dt}{t^{1/n}\cdot (h^*)^{-1}(t)}<\infty.
\end{equation}
It is a stronger condition than (\ref{integrability condition: 1}). Actually, by the trivial inequality $h^{-1}(t)\cdot (h^*)^{-1}(t)\le 2t$ for all $t>h^*(\tau_0)$, we have
\begin{eqnarray*}
\int_{h(\tau_0)}^\infty\frac{dt}{t^{1/n}\cdot (h^*)^{-1}(t)}&\ge&\frac{1}{2}\cdot\int_{h(\tau_0)}^\infty\frac{h^{-1}(t)dt}{t^{1+\frac{1}{n}}}\\
&=&\frac{1}{2}\cdot\int_{\tau_0}^\infty\frac{s\cdot h'(s)ds}{h(s)^{1+\frac{1}{n}}}\\
&=&\frac{n}{2}\cdot\frac{\tau_0}{h(\tau_0)^{1/n}}+\frac{n}{2}\cdot\int_{\tau_0}^\infty\frac{ds}{h(s)^{\frac{1}{n}}},
\end{eqnarray*}
where we used the variable change $h^{-1}(t)=s$. On the other hand, if $h$ is convex for large $t$, namely $t\ge\tau_1$ for some $\tau_1>0$, then the two integrability conditions are equivalent.

Assume that $h(t)$ is convex for large $t$, namely $t\ge \tau_1$ for some large constant $\tau_1$. Then $h$ is the principal part of some $N$-function, namely an $N$-function $\hat{h}$ so that $\hat{h}=h$ on the interval $[\tau_2,\infty)$ for some larger $\tau_2$. See \cite[Page 16-17]{KrRu} for a construction of such $H$. Then, by properness, $h^*(s)=\hat{h}(s)$ for sufficiently large $s$. In particular, there is $\tau_3$ so that $h(s)=\hat{h}(s)$ and $h^*(s)=\hat{h}^*(s)$ when $s\ge\tau_3$. By properness of $h$ again,
$$(h^*)^{-1}(s)=(\hat{h}^*)^{-1}(s)>\frac{s}{\hat{h}^{-1}(s)}=\frac{s}{h^{-1}(s)},$$
whenever $s\ge\tau_4$ for some sufficiently large $\tau_4$. Then, for any $\tau\ge\tau_4$,
\begin{eqnarray*}
\int_{\tau}^\infty\frac{dt}{t^{1/n}\cdot(h^*)^{-1}(t)}&<&\int_\tau^\infty\frac{h^{-1}(t)dt}{t^{1+\frac{1}{n}}}\\
&=&\int_{h^{-1}(\tau)}^\infty\frac{s\cdot h'(s)ds}{h(s)^{1+\frac{1}{n}}}\\
&=&n\cdot\frac{h^{-1}(\tau)}{\tau^{\frac{1}{n}}}+n\cdot\int_{h^{-1}(\tau)}^\infty\frac{ds}{h^{1/n}(s)}.
\end{eqnarray*}
Together with the discussion above, we see that the integrable conditions (\ref{integrability condition: 1}) and (\ref{integrability condition: 2}) are equivalent. In Section \ref{example section} we show that all the $h$ functions in the interested examples are all bounded below by certain $N$-functions.


\begin{definition}\label{Definition-hbar}
Under the integrability condition (\ref{integrability condition: 2}) we define
\begin{equation}
H(t)=H_\Phi(t)=\int_t^\infty\frac{ds}{s^{1/n}\cdot (h^*)^{-1}(s)},
\end{equation}
for any $t\ge h(\tau_0)$. Moreover, for any small $\delta>0$, we define $\tau(\delta)$ to be the time $\tau$ so that
\begin{equation}
\delta\cdot\Phi^*(\tau)=H(\tau).
\end{equation}
Then we put
\begin{equation}
\hbar(\delta)=H(\tau(\delta)).
\end{equation}
\end{definition}

By the monotonicity of $\hbar(\delta)$ and $\Phi^*$, we can easily check that the function $\tau(\delta)$ is monotone decreasing with respect to $\delta.$ The function $H(t)$ decreases strictly to $0$ as $t\rightarrow\infty$. Then, by monotonicity, one sees that $\tau(\delta)$ is uniquely determined for any $\delta\le\frac{H(h(\tau_0))}{\Phi^*(h(\tau_0))}$. It can also be seen trivially that
$$\lim_{\delta\rightarrow 0}\tau(\delta)=\infty,~~\mbox{and}~~\lim_{\delta\rightarrow 0}\hbar(\delta)=0$$
In particular, the function
\begin{equation}\label{Phi-large-enough}
\Phi^*(\tau(\delta))=\frac{\hbar(\delta)}{\delta}\gg 1, \hspace{0.5cm}\forall~0<\delta\ll 1.
\end{equation}
The functions $\tau$ and $\hbar$ will relates crucially to the stability of plurisubharmonic functions around a solution to CMA (\ref{MA: smooth case}).

\subsection{Growth conditions for $N$-function}\label{construction-of-N-1}

In order to obtain the geometric estimates such as the Green function estimate and the diameter estimate to solutions of CMA (\ref{MA: smooth case}), more restrictions to the $N$-function are needed. The first growth condition for $N$-function $\Phi$ that will be used in the paper is the following
\begin{equation}\label{growth condition for N-function: 1}
\frac{\Phi(t_1\cdot t_2)}{t_1\cdot t_2}\le L\cdot\left(\frac{\Phi(t_1)}{t_1}+\frac{\Phi(t_2)}{t_2}\right),\qquad \forall ~~t_1,t_2\ge 1,
\end{equation}
for some constant $L>0$. The condition implies in particular that $\Phi(t^2)\le 2L\cdot\Phi(t)\cdot t$. Due to the argument in Chapter I \S 6 of \cite{KrRu}, if an $N$-function satisfies (\ref{growth condition for N-function: 1}), then its complementary $\Phi^*$ satisfies $\triangle^2$-condition, so
$$\Phi(t)\le t\cdot\log^\beta t,$$
for some large constant $\beta$ and any sufficiently large $t$. On the other hand, it is easy to check that any slow growth $N$-function in section \ref{slow example} satisfies the condition (\ref{growth condition for N-function: 1}).



\section{Complex Monge-Amp\`ere equation in Orlicz space}\label{CMA-in-Orlicz}

Let $X$ be a compact K\"ahler manifold and $\omega_0$ be a fixed background K\"ahler metric on $X$. Let $\theta$ be a smooth (1,1) form in a big class as in the Introduction. Tian's $\alpha$-invariant \cite{Ti} gives a uniform Moser-Trudinger inequality for the measure $\omega_0^n$, namely,
\begin{equation}\label{alpha invariant}
\frac{1}{[\omega_0^n]}\cdot\int_X e^{-\delta(\psi-\sup\psi)}\omega_0^n\le C
\end{equation}
for some positive constants $\delta=\delta(\omega_0,A)$, $C=C(\omega_0,A)$ and any $\psi\in\Psh(\theta)$.

Suppose that $\varphi_0\in\Psh(\theta)$ satisfies a CMA (\ref{MA: smooth case}) and normalization (\ref{prob measure}). Assume further that $F$ belongs to an Orlicz space $L^\Phi(\omega_0^n)$ for some $N$-function $\Phi$. In this section we give a proof of the a priori $L^\infty$ estimate to the CMA. The basic reference is Ko{\l}odziej classical paper \cite{Ko98}; more references are \cite{Ko98, Ko02, Ko03, Ko05} and \cite{EyGuZe09, FGS20, GuZe12, PhSoSt12, GuPh23, GuPhTo23-2}.

We will take Ko{\l}odziej approach \cite{Ko98} and follow the arguments by Fu-Guo-Song \cite{FGS20} to prove Theorem \ref{infinity estimate: 2} and the argument of Guo-Phong-Tong-Wang \cite{GuPhToWa21} to prove Theorem \ref{Stability theorem 1}.

\subsection{Measure domination by capacity}

Let $\mathcal{V}_\theta$ be the envelope (or extremal function) of $\theta$,
\begin{equation}
\mathcal{V}_\theta(x)=\sup\{\psi(x)\,|\,\psi\in\Psh(\theta),\,\sup\psi=0.\}
\end{equation}
which is $\theta$-plurisubharmonic. By \cite{Be19, BeDe12, To18, ChZhou19} the envelope $\mathcal{V}_\theta$ has $C^{1,1}$ regularity. There is a "tautological maximum  principle" \cite[P1028]{GuZe12},
\begin{equation}
\sup\psi=\sup(\psi-\mathcal{V}_\theta),
\end{equation}
for any $\psi\in\Psh(\theta)$.

More generally, for any open subset $E$ there is a definition of envelope \cite{BEGZ10, EyGuZe09, FGS20}
$$\mathcal{V}_{\theta, E}=\sup\{u\in\Psh(\theta)\,|\,u\le 0\mbox{ in }E\}.$$
is called the {\it extremal function} of $E$. The extremal function $\mathcal{V}_{\theta, E}$ satisfies
\begin{itemize}
\item[(i)] $\mathcal{V}_{\theta, E}\in\Psh(\theta)\cap L^\infty(X)$;
\item[(ii)] $\mathcal{V}_{\theta, E}\le0$ on $E$;
\item[(iii)] $\big(\theta+\sqrt{-1}\partial\bar{\partial}\mathcal{V}_{\theta, E}\big)^n=0$ on $X\backslash\overline{E}$.
\end{itemize}

Let $\varphi_0$ be a solution to the CMA {\rm(\ref{MA: smooth case})}. Define $\omega=\theta+\sqrt{-1}\partial\bar{\partial}\varphi_0$.

\begin{definition}[Relative Capacity \cite{BEGZ10, FGS20}]
For any Borel measurable subset $K$ define the capacity
\begin{equation}
\Ca_{\theta}(K)=\sup\bigg\{\int_K\big(\theta+\sqrt{-1}\partial\bar{\partial}u\big)^n\bigg|\,u\in\Psh(\theta),\,0\le u-\mathcal{V}_{\theta}\le 1\bigg\}.
\end{equation}
\end{definition}

Two a prior estimates are in order. The first lemma is about a capacity estimate, taken from \cite[Lemma 3.2]{FGS20}; see \cite[Corollary 2.3, Lemma 4.4]{BEGZ10} etal for other versions of this kind of estimates. The second one is a measure domination property for Monge-Amp\`ere measure in Orlicz space; it is a slight generalization of previous results in \cite{Ko98, FGS20, BEGZ10}.

\begin{lemma}{\rm\cite{Ko98, BEGZ10, FGS20}}
Let $\psi\in\Psh(\theta)\cap \mathcal{E}(X)$. For any $s\in\mathbb{R}$ and $0<r\le 1$ we have
\begin{equation}\label{capacity estimate: 11}
r^n\cdot\Ca_{\theta}\{\psi-\mathcal{V}_{\theta}<s-r\}\le\int_{\{\psi-\mathcal{V}_{\theta}<s\}}\big(\theta+\sqrt{-1}\partial\bar{\partial}\psi\big)^n
\end{equation}
\end{lemma}

\begin{lemma}\label{measure estimate: 0}
For any open set $E$ we have
\begin{equation}\label{measure estimate: 1}
\frac{1}{[\omega_0^n]}\int_E F\omega_0^n\le C(\omega_0,A)\cdot\frac{\|F\|_{L^\Phi(\omega_0^n)}}{h\left(\delta\cdot\left(\frac{\Ca_{\theta}(E)}{[\omega^n]}\right)^{-1/n}\right)},
\end{equation}
where $\delta=\delta(\omega_0)$ is a positive constant depending only on $\omega_0$.
\end{lemma}
\begin{proof}
Let $\delta>$ be a small positive constant that will be determined later. Put $v=e^{-\delta(\mathcal{V}_{\theta,E}-\sup\mathcal{V}_{\theta,E}-1)}$. Then, for any $b>\|F\|_{L^\Phi}$,
$$(\Phi^*)^{-1}(v)\cdot \frac{F}{b}\le\Phi\left(\frac{F}{b}\right)+v.$$
Using that $\mathcal{V}_{\theta,E}\le 0$ on $E$ and integrating over $E$ we get
\begin{eqnarray*}
\frac{(\Phi^*)^{-1}(e^{\delta\cdot(\sup\mathcal{V}_{\theta,E}+1))})}{b}\cdot\frac{1}{[\omega_0^n]}\cdot\int_E F\omega_0^n&\le&\frac{1}{[\omega_0^n]}\cdot\int_E (\Phi^*)^{-1}(v)\cdot \frac{F}{b}\omega_0^n\\
&\le&1+\frac{1}{[\omega_0^n]}\cdot\int_X e^{-\delta(\mathcal{V}_{\theta,E}-\sup\mathcal{V}_{\theta,E}-1)}\omega_0^n\\
&\le& C(\omega_0,A).
\end{eqnarray*}
Noticing that $\mathcal{V}_{\theta,E}\in\Psh(A\omega_0)$, one can take one $\delta$ smaller than that in (\ref{alpha invariant}). Then, taking infimum over $b$ we get
$$\frac{1}{[\omega_0^n]}\cdot\int_E F\omega_0^n\le C(\omega_0,A)\cdot\|F\|_{L^\Phi}\cdot \frac{1}{h(\delta\cdot(\sup\mathcal{V}_{\theta,E}+1))}.$$
Finally, applying the argument in \cite[Lemma 3.1]{FGS20} to get the lower bound of $\sup\mathcal{V}_{\theta,E}$. If $\sup\mathcal{V}_{\theta,E}\ge 1$,
$$\sup\mathcal{V}_{\theta,E}\ge \left(\frac{\Ca_{\theta}(E)}{[\omega^n]}\right)^{-1/n};$$
otherwise,
$$\sup\mathcal{V}_{\theta,E}\ge\left(\frac{\Ca_{\theta}(E)}{[\omega^n]}\right)^{-1/n}-1.$$
In either case we can conclude the required estimate.
\end{proof}

\subsection{$L^\infty$-estimate for CMA}

We give a sketched proof of Theorem \ref{infinity estimate: 2} following the arguments in Fu-Guo-Song \cite{FGS20}.

First of all, introduce two families of constants
\begin{equation}
\vartheta(s)=\bigg(\frac{\Ca_{\theta}\{\varphi_0-\mathcal{V}_\theta<-s\}}{[\omega^n]}\bigg)^{1/n},
\end{equation}
and
\begin{equation}
\varrho(s)=\frac{\vol_\omega\{\varphi_0-\mathcal{V}_\theta<-s\}}{[\omega^n]}=\frac{1}{[\omega_0^n]}\int_{\{\varphi_0-\mathcal{V}_\theta<-s\}}F\omega_0^n.
\end{equation}
The capacity estimate (\ref{capacity estimate: 11}) gives a universal estimate
\begin{equation}\label{capacity estimate: 14}
r\cdot\vartheta(s+r)\le \varrho^{1/n}(s),\hspace{0.5cm}\forall~s\ge 0,~~\mbox{ and }1\ge r\ge 0.
\end{equation}
Combining with (\ref{measure estimate: 1}) we get
\begin{equation}\label{capacity estimate: 15}
r\cdot\vartheta(s+r)\le C(\omega_0,A)\cdot\|F\|_{L^\Phi(\omega_0)}^{1/n}\cdot \frac{1}{h^{1/n}\left(\delta\cdot\vartheta(s)^{-1}\right)},\hspace{0.5cm}\forall~s\ge 0,\,\mbox{ and }1\ge r\ge 0.
\end{equation}

\begin{lemma}\label{capacity estimate: lemma}
There is a universal estimate, for any $s>0$,
\begin{equation}\label{capacity estimate: 16}
\varrho(s)\le C(\omega_0,A)\cdot\|F\|_{L^\Phi(\omega_0)}\cdot\frac{1}{\left(\Phi^*\right)^{-1}(s)}.
\end{equation}
In particular, $\varrho(s)\rightarrow 0$ and $\vartheta(s)\rightarrow 0$ as $s\rightarrow\infty$.
\end{lemma}
\begin{proof}
By direct calculation, for any $\epsilon>0$ and $b>\|F\|_{L^\Phi}$,
$$F=\epsilon b\cdot\frac{F}{\epsilon b}\le\epsilon b\cdot\left(\Phi(F/b)+\Phi^*(\epsilon^{-1})\right).$$
Then, for any $s>0$,
\begin{eqnarray*}
\varrho(s)=\frac{1}{[\omega_0^n]}\int_{\{\varphi_0-\mathcal{V}_\theta<-s\}}F\omega_0^n
\le\frac{\epsilon b}{[\omega_0^n]}\int_{\{\varphi_0-\mathcal{V}_\theta<-s\}}\left(\Phi(F/b)+\Phi^*(\epsilon^{-1})\right)\omega_0^n.
\end{eqnarray*}
Taking infimum over $b$ gives
\begin{eqnarray*}
\varrho(s)\le \|F\|_{L^\Phi}\cdot\epsilon\cdot\left(1+\Phi^*(\epsilon^{-1})\cdot\frac{1}{[\omega_0^n]}\cdot\int_{\{\varphi_0-\mathcal{V}_\theta<-s\}}\omega_0^n\right),
\end{eqnarray*}
where
\begin{eqnarray*}
\frac{1}{[\omega_0^n]}\cdot\int_{\{\varphi_0-\mathcal{V}_\theta<-s\}}\omega_0^n&\le&\frac{1}{s}
\cdot\frac{1}{[\omega_0^n]}\cdot\int_{\{\varphi_0-\mathcal{V}_\theta<-s\}}(\mathcal{V}_{\theta}-\varphi_0)\omega_0^n\\
&\le&\frac{C(\omega_0,A)}{s}
\cdot\frac{1}{[\omega_0^n]}\cdot\int_{\{\varphi_0-\mathcal{V}_\theta<-s\}}e^{\delta(\mathcal{V}_{\theta}-\varphi_0)}\omega_0^n\\
&\le&\frac{C(\omega_0,A)}{s}
\cdot\frac{1}{[\omega_0^n]}\cdot\int_{\{\varphi_0-\mathcal{V}_\theta<-s\}}e^{\delta(\sup\varphi_0-\varphi_0)}\omega_0^n\\
&\le&\frac{C(\omega_0,A)}{s},
\end{eqnarray*}
where we used that $\sup\varphi_0=\sup(\varphi_0-\mathcal{V}_\theta)=0$ and that $\mathcal{V}_\theta\le 0$ over $X$. Here $\delta=\delta(\omega_0)$ is any positive constant smaller than that in (\ref{alpha invariant}). Substituting into the previous formula gives
$$\varrho(s)\le\epsilon\cdot\|F\|_{L^\Phi}\cdot\left(1+C(\omega_0,A)\cdot\frac{\Phi^*(\epsilon^{-1})}{s}\right).$$
Taking $\epsilon$ so that $\Phi^*(\epsilon^{-1})=s$ we get (\ref{capacity estimate: 16}).
\end{proof}

In the following, we will apply above estimates to the De Giorgi iteration. By (\ref{capacity estimate: 15}), there are positive constants $C_0=C_0(\omega_0,A)$ and $\delta=\delta(\omega_0,A)$ such that, for any $s>0$ and $0< r\le 1$,
\begin{eqnarray*}
\vartheta(s+r)&\le& C_0\cdot\|F\|_{L^\Phi(\omega_0)}^{1/n}\cdot r^{-1}\cdot \frac{1}{h^{1/n}\left(\delta\cdot\vartheta(s)^{-1}\right)}\\
&\le&\frac{\vartheta(s)}{2}\cdot r^{-1}\cdot C_1\cdot\|F\|_{L^\Phi(\omega_0)}^{1/n}\cdot\frac{\delta\cdot\vartheta(s)^{-1}}{h^{1/n}\left(\delta\cdot\vartheta(s)^{-1}\right)}
\end{eqnarray*}
for some $C_1=C_1(\omega_0,A)$. It follows that
\begin{equation*}
\vartheta(s+r)\le\frac{\vartheta(s)}{2}
\end{equation*}
whenever the number
\begin{equation*}
r=C_1\cdot\|F\|_{L^\Phi(\omega_0)}^{1/n}\cdot\frac{\delta\cdot\vartheta(s)^{-1}}{h^{1/n}\left(\delta\cdot\vartheta(s)^{-1}\right)}\le 1.
\end{equation*}
By the monotonicity of $h$
\begin{equation*}
\frac{t}{h^{1/n}(t)}\le 2\int_{t/2}^t\frac{ds}{h^{1/n}(s)}.
\end{equation*}
The assumption (\ref{integrability 0}) shows that the left hand side tends to zero as $t\rightarrow\infty$. Take some big constant $s_0=s_0(\omega_0,A,\Phi,\|F\|_{L^\Phi})$ so that
\begin{equation}\label{e4-10}
2C_1\cdot\|F\|_{L^\Phi(\omega_0)}^{1/n}\cdot\int_{\frac{\delta\cdot\vartheta(s_0)^{-1}}{2}}^\infty\frac{ds}{h^{1/n}(s)}\le 1.
\end{equation}
Then, define a sequence $s_j$ by induction
\begin{equation*}
s_{j+1}=\inf\left\{s_i+r\,\bigg|\,\vartheta(s_j+r)\le\frac{1}{2}\vartheta(s_j)\right\}
\end{equation*}
and put $t_j=\delta\cdot\vartheta(s_j)^{-1}=2^{j}t_0$. Then,
\begin{equation*}
s_{j+1}-s_j\le C_1\cdot\|F\|_{L^\Phi(\omega_0)}^{1/n}\cdot\frac{t_j}{h^{1/n}(t_j)}.
\end{equation*}
Thus, by assumption (\ref{e4-10}),
\begin{eqnarray}\label{Kolo assumption: 2}
\sum_{j=0}^\infty(s_{j+1}-s_j)\le 2C_1\cdot\|F\|_{L^\Phi(\omega_0)}^{1/n}\cdot\int_{\frac{t_0}{2}}^\infty\frac{dt}{h^{1/n}(t)}.
\end{eqnarray}

Finally, by a variable change $t=\log s$, the integral on the right hand side
$$\int_{\frac{t_0}{2}}^\infty\frac{dt}{h^{1/n}(t)}=\int_{\frac{\log s_0}{2}}^\infty\frac{ds}{s\cdot \left[(\Phi^*)^{-1}(s)\right]^{1/n}}\le \mathbf{I},$$
where $s_0=e^{t_0}$ is a constant depending on $\omega_0,A,\Phi,n,\|F\|_{L^\Phi}$. It follows that $\vartheta(s)=0$ whenever $s>s_0+2C_1\cdot\|F\|_{L^\Phi(\omega_0)}^{1/n}\cdot\mathbf{I}$. The proof of the Theorem \ref{infinity estimate: 2} is complete.\qed

\subsection{Stability estimate of psh functions around CMA}\label{proof of stability of cma}

In this subsection we shall prove Theorem \ref{Stability theorem 1}, a stability property for plurisubharmonic (psh for short) functions around $\varphi_0$, which solves the CMA (\ref{MA: smooth case}). All the constants $C_i$ in the proof depend only on $\omega_0,A,\Phi,n,K$. We may assume that $K\ge 1$.

Let $c>0$ be a constant to be determined later. Define the subset
$$\Omega_s=\{\psi_\delta>\varphi_0+c+s\}.$$
We are going to show that $\Omega_s=\emptyset$ whenever $s$ is sufficiently large. The argument is divided into several steps.

{\bf Step 1.} We have
\begin{equation}\label{e5-09}
\frac{1}{[\omega^n]}\int_{\Omega_0}\omega^n\le \frac{2K}{(\Phi^*)^{-1}\left(\frac{c}{\delta}\right)}.
\end{equation}

Same as in the proof of formula (\ref{capacity estimate: 16}). First of all, for any $b>\|F\|_{L^\Phi}$, on $\Omega_0$,
$$vF\le \Phi(F/b)+\delta^{-1}\cdot(\psi_\delta-\varphi_0)$$
where $v=b^{-1}\cdot (\Phi^*)^{-1}\left(\frac{\psi_\delta-\varphi_0}{\delta}\right)$. Integrating over $\Omega_0$ gives
$$\frac{1}{[\omega_0^n]}\cdot\int_{\Omega_0}vF\omega_0^n\le 1+\frac{1}{[\omega_0^n]}\int_X\frac{(\psi_\delta-\varphi_0)_+}{\delta}\omega_0^n\le 2.$$
Taking infimum over $b>\|F\|_{L^\Phi}$ gives
$$\frac{1}{[\omega_0^n]}\cdot\int_{\Omega_0}(\Phi^*)^{-1}\left(\frac{\psi_\delta-\varphi_0}{\delta}\right)\cdot F\omega_0^n\le 2\|F\|_{L^\Phi}.$$
Noticing that the function $\psi_\delta-\varphi_0\ge c$ on $\Omega_0$, together with the monotonicity of $\Phi^*$, we get the required estimate.

{\bf Step 2.} By approximation we may assume that all the functions in the proof are smooth; see \cite{GuPhToWa21,GuPhToWa24} for a full discussion to the approximation process. Let $\phi_s\in\Psh(\theta)$ be the solution to the modify CMA, for $s\ge 0$,
\begin{equation}
\frac{\left(\theta+\sqrt{-1}\partial\bar{\partial}\phi_s\right)^n}{[\omega^n]}=\frac{\left(\psi_\delta-\varphi_0-c-s\right)_+}{B_s}\cdot F\cdot\frac{\omega_0^n}{[\omega_0^n]},
\end{equation}
with
$$\sup(\phi_s-\mathcal{V}_\theta)=\sup\phi_s=0,$$
where $B_s$ is a normalizing constant
$$B_s=\frac{1}{[\omega_0^n]}\int_X\left(\psi_\delta-\varphi_0-c-s\right)_+\cdot F\cdot\omega_0^n.$$

On the set $\Omega_s$ we construct the test function
\begin{equation*}
H=(\varphi_0-\phi_s)-\frac{1}{2(n+1)}\cdot\frac{\left(\psi_\delta-\varphi_0-c-s\right)^{\frac{n+1}{n}}}{B_s^{\frac{1}{n}}}.
\end{equation*}
It satisfies
\begin{eqnarray*}
\triangle_\omega H&\le& n-\tr(\theta+\sqrt{-1}\partial\bar{\partial}\phi_s)-\frac{1}{2}\cdot\left(\frac{\psi_\delta-\varphi_0-c-s}{B_s}\right)^{1/n}\\
&\le&n-\frac{1}{2}\cdot\left(\frac{\psi_\delta-\varphi_0-c-s}{B_s}\right)^{1/n};
\end{eqnarray*}
and
$$H\ge -\|\varphi_0-\mathcal{V}_\theta\|_{\infty},\hspace{0.5cm}\mbox{on}~\partial\Omega_s.$$
By the maximum principle we obtain
$$H\ge-C_0-C(n)\cdot B_s,\hspace{0.5cm}\mbox{in}~\Omega_s.$$
It follows that, on the domain $\Omega_s$,
\begin{eqnarray}
\frac{\left(\psi_\delta-\varphi_0-c-s\right)^{\frac{n+1}{n}}}{B_s^{\frac{1}{n}}}&\le& 2(n+1)\cdot\left[\left(\varphi_0-\phi_s\right)+C_0\right]+C(n)B_s\nonumber\\
&\le&2(n+1)\cdot\left(-\phi_s\right)+C_1+C(n)B_s.\label{e5-10}
\end{eqnarray}

{\bf Step 3.} Bound $B_s$ uniformly as follows, for any $s>0$,
\begin{eqnarray*}
B_s&=&\frac{1}{[\omega_0^n]}\int_X\left(\psi_\delta-\varphi_0-c-s\right)_+\cdot F\cdot\omega_0^n\\
&\le&\frac{1}{[\omega_0^n]}\int_{\Omega_s}\left(\psi_\delta-\varphi_0\right)_+\cdot F\cdot\omega_0^n\\
&\le&\frac{1}{[\omega_0^n]}\int_{\Omega_s}\left(\mathcal{V}_\theta-\varphi_0\right)\cdot F\cdot\omega_0^n\\
&\le&\frac{C_0 }{[\omega_0^n]}\int_{\Omega_s} F\cdot\omega_0^n\\
&\le&C_0,
\end{eqnarray*}
where we used that $\psi_\delta\le\mathcal{V}_\theta$. 

{\bf Step 4.} There is a Moser-Trudinger inequality, by (\ref{e5-10}),
\begin{equation}
\frac{1}{[\omega_0^n]}\int_{\Omega_s}\exp\left\{c_0\frac{\left(\psi_\delta-\varphi_0-c-s\right)^{\frac{n+1}{n}}}{B_s^{\frac{1}{n}}}\right\}\omega_0^n\le C_3,
\end{equation}
for some $c_0>0$ depending on the $\alpha$-invariant of $\omega_0$. Put
$$\xi(s)=\frac{\left(\psi_\delta-\varphi_0-c-s\right)^{\frac{n+1}{n}}}{B_s^{\frac{1}{n}}}.$$
Then, compute as in Step 1,
\begin{equation}\label{e5-11}
\frac{1}{[\omega_0^n]}\int_{\Omega_s}h(c_0\cdot\xi(s))\cdot F\omega_0^n\le (C_3+1)\cdot\|F\|_{L^\Phi}.
\end{equation}
where $h(t)=(\Phi^*)^{-1}(e^t)$ as defined before.

{\bf Step 5.} Improve the upper bound of $B_s$. Introduce the volume function
$$\varrho(s)=\frac{1}{[\omega_0^n]}\int_{\Omega_s}F\omega_0^n.$$
(\ref{e5-09}) shows that $\varrho(s)\le\varrho(0)$ can be sufficiently small.

Applying the conjugate $h^*$ to get
$$c_0\xi=\epsilon\cdot c_0\xi\cdot\frac{1}{\epsilon}\le\epsilon\cdot\left(h(c_0\xi)+h^*\left(\frac{1}{\epsilon}\right)\right).$$
Integrating over $\Omega_s$ gives, by (\ref{e5-11}),
\begin{equation*}
\frac{1}{[\omega_0^n]}\int_{\Omega_s}\xi\cdot F\omega_0^n\le\frac{\epsilon}{c_0}\cdot\left((C_3+1)\cdot\|F\|_{L^\Phi}+h^*\left(\frac{1}{\epsilon}\right)\cdot\varrho(s)\right).
\end{equation*}
Taking $\epsilon$ so that
$$h^*\left(\frac{1}{\epsilon}\right)\cdot\varrho(s)=4K.$$
Notice that $\varrho(s)\le\Phi^*(1)\cdot\|F\|_{L^\Phi}$. After justifying a larger constant depending on $\Phi^*(1)$ which does not effect the argument below, we may assume that $4K/\varrho(s)$ is big enough so that $h^*$ can be defined. We then conclude that
\begin{equation}\label{e5-12}
\frac{1}{[\omega_0^n]}\int_{\Omega_s}\xi(s)\cdot F\omega_0^n\le \frac{C_4}{(h^*)^{-1}\left(\frac{4K}{\varrho(s)}\right)}.
\end{equation}
By the definition of $\xi(s)$ and the H\"older inequality,
$$B_s\le\left(\frac{1}{[\omega_0^n]}\int_{\Omega_s}B_s^{\frac{1}{n}}\cdot\xi(s)\cdot F\omega_0^n\right)^{\frac{n}{n+1}}\cdot \varrho^{\frac{1}{n+1}}$$
It follows
\begin{equation}\label{e5-13}
B_s\le C_4\cdot\frac{\varrho(s)^{\frac{1}{n}}}{(h^*)^{-1}\left(\frac{4K}{\varrho(s)}\right)}.
\end{equation}

{\bf Step 6.} Iteration formula
\begin{equation}\label{e5-14}
r\varrho(r+s)\le C_4\cdot\frac{\varrho(s)^{\frac{1}{n}}}{(h^*)^{-1}\left(\frac{4K}{\varrho(s)}\right)},\hspace{0.5cm}\forall ~~ r,s\ge 0,
\end{equation}
which follows by a direct calculation
\begin{eqnarray*}
r\varrho(r+s)=\frac{r}{[\omega_0^n]}\int_{\Omega_{s+r}}F\omega_0^n
\le\frac{1}{[\omega_0^n]}\int_{\Omega_{s+r}}(\psi_\delta-\varphi_0-c-s)\cdot F\omega_0^n
\le B_s.
\end{eqnarray*}

{\bf Step 7.} $\varrho(s)$ vanishes for certain controllable $s>0$.

Argue as in the proof of the $L^\infty$-estimate, Theorem \ref{infinity estimate: 2}. Rewrite the iteration formula (\ref{e5-14}) as
$$\varrho(s+r)\le\frac{\varrho(s)}{2}\cdot 2C_4\cdot (4K)^{-\frac{n-1}{n}}\cdot r^{-1}\cdot\frac{\left(\frac{4K}{\varrho(s)}\right)^{\frac{n-1}{n}}}{(h^*)^{-1}\left(\frac{4K}{\varrho(s)}\right)}$$
which is less than $\varrho(s)/2$ whenever $r$ is taken to be
\begin{equation*}
r= 2C_4\cdot(4K)^{-\frac{n-1}{n}}\cdot\frac{\left(\frac{4K}{\varrho(s)}\right)^{\frac{n-1}{n}}}{(h^*)^{-1}\left(\frac{4K}{\varrho(s)}\right)}
=C_5\cdot \frac{\left(\frac{4K}{\varrho(s)}\right)^{\frac{n-1}{n}}}{(h^*)^{-1}\left(\frac{4K}{\varrho(s)}\right)}.
\end{equation*}

We start the iteration at $s_0=0$. Suppose $\varrho(0)\neq 0$. Define a sequence $s_j$ by induction
\begin{equation}
s_{j+1}=\inf\left\{s_i+r\,\bigg|\,\varrho(s_j+r)\le\frac{1}{2}\varrho(s_j)\right\}.
\end{equation}
Obviously, $\varrho(s_j)\neq 0$ for any $j$, and
\begin{equation}
s_{j+1}-s_j\le C_5\cdot\frac{\left(\frac{4K}{\varrho(s_j)}\right)^{\frac{n-1}{n}}}{(h^*)^{-1}\left(\frac{4K}{\varrho(s_j)}\right)}.
\end{equation}
Define a sequence of times $t_j=4K/\varrho(s_j)=2^j\cdot t_0$. By the monotonicity of $h^*$,
\begin{equation*}
\frac{\left(\frac{4K}{\varrho(s_j)}\right)^{\frac{n-1}{n}}}{(h^*)^{-1}\left(\frac{4K}{\varrho(s_j)}\right)}
=\frac{t_j^{\frac{n-1}{n}}}{(h^*)^{-1}\left(t_j\right)}
\le 2\int_{t_j/2}^{t_j}\frac{dr}{t^{1/n}\cdot (h^*)^{-1}(t)}
\end{equation*}
for any $j\ge 0$. Thus,
\begin{eqnarray*}
\sum_{j=0}^\infty(s_{j+1}-s_j)\le 2C_5\cdot\int_{\frac{t_0}{2}}^\infty\frac{dt}{t^{1/n}\cdot (h^*)^{-1}(t)}.
\end{eqnarray*}
In particular, $\varrho(s)=0$, for any $s\ge 2C_5\cdot\int_{\frac{t_0}{2}}^\infty\frac{dt}{t^{1/n}\cdot (h^*)^{-1}(t)}$.

{\bf Step 8.} Choices of $c$ and $\delta$.

From the discussion we have that
\begin{equation*}
\psi_\delta\le\varphi_0+c+2C_5\cdot\int_{\frac{t_0}{2}}^\infty\frac{dt}{t^{1/n}\cdot (h^*)^{-1}(t)}
\end{equation*}
where, by (\ref{e5-09}),
\begin{equation*}
t_0=\frac{4K}{\varrho(0)}\ge 2\cdot(\Phi^*)^{-1}\left(\frac{c}{\delta}\right).
\end{equation*}
It follows that
\begin{equation*}
\psi_\delta\le\varphi_0+c+2C_5\cdot\int_{(\Phi^*)^{-1}\left(\frac{c}{\delta}\right)}^\infty\frac{dt}{t^{1/n}\cdot (h^*)^{-1}(t)}
\end{equation*}
Put $\tau=(\Phi^*)^{-1}\left(\frac{c}{\delta}\right)$. Then,
\begin{equation}\label{modulus estimate of CMA: 00}
\psi_\delta\le\varphi_0+\delta\cdot\Phi^*(\tau)+2C_5\cdot\int_{\tau}^\infty\frac{dt}{t^{1/n}\cdot (h^*)^{-1}(t)}
\end{equation}
According to Definition \ref{Definition-hbar}, then we get the required estimate (\ref{stability estimate}) by picking $C=2C_5+1$.\qed



\section{Green's function and geometric bound}\label{Green function}

Let $X$ be a K\"ahler manifold of dimension $n\ge 2$.

Let $\varphi_0$ be a solution to the CMA {\rm(\ref{MA: smooth case})} where $F\in L^\Phi$ for some $N$-function which satisfies the first integrability condition (\ref{integrability 0}). Let $\omega=\theta+\sqrt{-1}\partial\bar{\partial}\varphi_0$. We assume that $\omega$ is a K\"ahler metric.

This section is devoted to proving the uniform estimate for the Green function associated with $\omega$. This is achieved by adapting the argument of Guo-Phong-Song-Sturm \cite{GuPhSoSt24-1} together with formal calculation in Orlicz space. The references for this section are \cite{GuPhSt24} and \cite{GuPhSoSt24-1}.

The constant $C$ in this section depend on $\omega_0,A,K,n,\Phi$ unless otherwise specified.

The following lemma is a natural extension of \cite[Lemma 2]{GuPhSt24} and \cite[Lemma 5.1]{GuPhSoSt24-1}. The proof in \cite[Lemma 2]{GuPhSt24}\cite[Lemma 5.1]{GuPhSoSt24-1} makes only use of the $L^{\infty}$-estimates, so we omit the details.

\begin{lemma}
Let $v\in L^1(X,\omega^n)$ be a function that satisfies $\int_Xv\omega^n=0$ and
$$v\in C^2(\overline{\Omega}_0),~~~\Delta_{\omega}v\geq-a~~in~~\Omega_0$$
for some $a>0$ and $\Omega_s=\{v>s\}$ is the super-level set of $v.$ Then there is a uniform constant $C>0$ such that
\begin{equation}
\sup_Xv\le C\cdot\big(a+\fint_X|v|\omega^n\big).
\end{equation}
\end{lemma}

Combining the lemma with the argument of \cite[Proposition 2.1]{GuPhSoSt24-2}, we have the following lemma.

\begin{lemma}\label{L-infty-for-Laplace}
Let $v\in C^2(X)$ be a function that satisfies $\int_Xv\omega^n=0$ and $|\Delta_{\omega}v|\le1$, then there is a uniform constant $C>0$ such that
\begin{equation}
\sup_Xv\le C.
\end{equation}
\end{lemma}

Let $G(x,y)$ be the Green's function of $(X,\omega)$. The following lemma is a consequence of the Lemma \ref{L-infty-for-Laplace}, see \cite[Lemma 5.4]{GuPhSoSt24-1} for a proof. It gives the basic Green function estimate in Orlicz space.

\begin{lemma}\label{L1-Green}
Assume as before, the Green's function satisfies
\begin{equation}
\int_X|G(x,\cdot)|\omega^n\le C,~~~and~~\inf_XG(x,\cdot)\ge-\frac{C}{[\omega^n]}
\end{equation}
for some $C>0$ independent of $x\in X$.
\end{lemma}

In principle more restrictions on $F$, or equivalently the $N$-function $\Phi$, gives more precise estimate of the Green function. In the following we will see how the integrability conditions affect the geometry of the K\"ahler metric $\omega$.

Define a modification $\mathcal{G}(x,y)$
$$\mathcal{G}(x,y)=G(x,y)-\inf_{x,y\in X}G(x,y)+\frac{1}{[\omega^n]}>0.$$
Fix a base point $x_0\in X$. Let $\mathcal{G}(y)=\mathcal{G}(x_0,y)$ and put $\widehat{\mathcal{G}}=[\omega^n]\cdot\mathcal{G}$. Then $\widehat{\mathcal{G}}\ge 1$.

\subsection{Refined integral estimate of Green function and diameter bound}\label{Refined integral estimate of Green function}

We begin with some technical assumptions. Let $\Phi_1$ be another $N$-function such that
\begin{itemize}
\item[(C1)] $\Phi_1$ is smaller than $\Phi$ in the sense that
\begin{equation}\label{growth condition for Phi 1}
\lim_{t\rightarrow\infty}\frac{\Phi(t)}{\Phi_1(t)}=\infty;
\end{equation}
\item[(C2)] the associated $h_1(t)=(\Phi_1^*)^{-1}(e^t)$ satisfies the first integrability condition (\ref{integrability 0});
\item[(C3)] the growth condition (\ref{growth condition for N-function: 1}) holds for some $L$, namely
\begin{equation}\label{growth condition for Phi 3}
\frac{\Phi_1(s\cdot t)}{s\cdot t}\le L\cdot\left(\frac{\Phi_1(s)}{s}+\frac{\Phi_1(t)}{t}\right),\qquad \forall ~~s,t\ge 1,
\end{equation}
for some $L>0$.
\end{itemize}

Due to the discussion in Section \ref{construction-of-N-1}, condition (\ref{growth condition for Phi 3}) implies that $\Phi_1(t)\le t(\log t)^\beta$ for some $\beta>0$ whenever $t$ is sufficiently large.

Let $E_1$ be increasing positive functions on $[1,\infty)$ such that
\begin{equation}\label{integrability comparison 2}
E_1(t)^{n-1}=\frac{t^2}{\Phi_1(t)},\qquad\forall ~~t\ge 1.
\end{equation}
By monotonicity of $\Phi_1(t)/t$, we have
\begin{equation*}
E_1(t)^{n-1}\le \frac{t}{\Phi_1(1)},\qquad\forall ~~t\ge 1.
\end{equation*}

Let $\Phi_2$ be another positive increasing function  on $[0,\infty)$ such that
\begin{equation}\label{integrability comparison 1}
\Phi_2\left(\frac{\Phi_1(t)}{t}\right)=\frac{\Phi(t)}{t},\qquad\forall~~ t\ge 1.
\end{equation}
Then
$$\lim_{s\rightarrow\infty}\frac{\Phi_2(s)}{s}=\lim_{t\rightarrow\infty}\frac{\Phi(t)}{\Phi_1(t)}=\infty.$$
So $\Phi_2^*(t)$ is well-defined for any $t\ge 0$. Define $E_2(t)$
\begin{equation}\label{integrability comparison 4}
E_2(t)=\inf\left\{s\,\bigg|\,\frac{\Phi_2^*(s^n)}{s}\ge t\right\}.
\end{equation}
Put $E(t)=\min(E_1(t),E_2(t))$ for any $t\ge 1$, which satisfies
\begin{equation}\label{integrability comparison 0}
E(t)^{n-1}\le\frac{t^2}{\Phi_1(t)}\le\frac{t}{\Phi_1(1)},
\end{equation}
and
\begin{equation}\label{integrability comparison 00}
\Phi_2^*\left(E(t)^n\right)\le t E(t).
\end{equation}

The function $E$ is related to the integrability of the Green function. Notice that it is uniquely determined by $\Phi_1$.

\begin{lemma}
For any $t\ge 1$,
\begin{equation}\label{integrability comparison 4.5}
\Phi_1\left(E(t)^n\right) \le C\cdot  E(t)\cdot t,
\end{equation}
for some constant $C$ depending only on the values $\Phi_1(1)$ and $L$.
\end{lemma}
\begin{proof}
By direct calculation, using monotonicity of $\Phi_1(t)/t$ and (\ref{integrability comparison 0}) and an iteration of (\ref{growth condition for Phi 3}),
\begin{eqnarray*}
\frac{\Phi_1\left(E(t)^n\right)}{E(t)^n} \le \frac{\Phi_1\left(E(t)\cdot t\cdot\frac{t}{\Phi_1(t)}\right)}{E(t)\cdot t\cdot\frac{t}{\Phi_1(t)}}
\le L^2\left(\frac{\Phi_1(E(t))}{E(t)}+\frac{\Phi_1(t)}{t}+\frac{\Phi_1\left(\frac{t}{\Phi_1(t)}\right)}{\frac{t}{\Phi_1(t)}}\right).
\end{eqnarray*}
Using that $E(t)\le t/\Phi_1(1)$ we get
$$\frac{\Phi_1(E(t))}{E(t)}\le L\left(\frac{\Phi_1(t)}{t}+\frac{\Phi_1\left(\frac{1}{\Phi_1(1)}\right)}{\frac{1}{\Phi_1(1)}}\right);$$
using monotonicity of $\Phi_1(t)/t$ again gives the bound of the third term on the right hand side
$$\frac{\Phi_1\left(\frac{t}{\Phi_1(t)}\right)}{\frac{t}{\Phi_1(t)}}\le \frac{\Phi_1\left(\frac{1}{\Phi_1(1)}\right)}{\frac{1}{\Phi_1(1)}}.$$
Multiplying $E(t)^n$ onto the first formula gives
$$\Phi_1\left(E(t)^n\right)\le L^2\cdot E(t)\cdot t\cdot \frac{t}{\Phi_1(t)}\cdot \left(C+2L^3\frac{\Phi_1(t)}{t}\right) \le C\cdot  E(t)\cdot t.$$
The proof is complete.
\end{proof}

\begin{proposition}\label{L1+epsilon-Green}
Assume as above. There exists $C=C(\omega_0,\Phi,\Phi_1,E; A,K,L,K,n)$, independent of $x_0$, such that
\begin{equation}\label{L1+-Green-function}
\int_X\mathcal{G}\cdot E(\widehat{\mathcal{G}})~\omega^n\le C.
\end{equation}
\end{proposition}
\begin{proof}
We follow closely the proof of \cite[Lemma 5.5]{GuPhSoSt24-1} and \cite[Proposition 4.1]{ZwZy25+}. By approximation we may assume that $J=\int_X\mathcal{G}\cdot E(\widehat{\mathcal{G}})~\omega^n$ is bounded. It is obvious that
$$J\ge E(1)\int_X\mathcal{G}\omega^n\ge E(1).$$
The aim is to show a uniform upper bound of $J$.

Define $\widehat{E}=\left(1+\frac{E(\widehat{\mathcal{G}})^n}{J}\right)^{1/n}$ and consider the auxiliary CMA
\begin{equation}\label{cma-psi-k}
\frac{(\theta+\sqrt{-1}\partial\bar{\partial}\psi)^n}{[\omega^n]}=\frac{\widehat{E}^n}{B}\cdot F\cdot\frac{\omega_0^n}{[\omega_0^n]},
\end{equation}
with $\sup_X\psi=0$, where $B$ is a constant. By (\ref{integrability comparison 0}), the principal term of the measure admits an upper bound
\begin{equation}\label{integrability comparison 6}
\widehat{E}^n=1+\frac{E(\widehat{\mathcal{G}})^n}{J}
\le 1+\frac{E(\widehat{\mathcal{G}})\cdot\widehat{\mathcal{G}}}{J}\cdot\frac{\widehat{\mathcal{G}}}{\Phi_1(\widehat{\mathcal{G}})}
\le 1+\frac{E(\widehat{\mathcal{G}})\cdot\widehat{\mathcal{G}}}{E(1)\Phi_1(1)}.
\end{equation}
The constant $B$ has an obvious lower bound $B\ge 1$. It also admits an upper bound
$$B=\fint_X \widehat{E}^n\cdot F\omega_0^n\le 1+\frac{1}{J\Phi_1(1)}\cdot\fint_X E(\widehat{\mathcal{G}})\cdot\widehat{\mathcal{G}}\cdot F\omega_0^n
=1+\frac{1}{\Phi_1(1)}.$$

We next show that the measure of the CMA admits an $L^{\Phi_1}$ bound. By (\ref{growth condition for Phi 3}) and (\ref{integrability comparison 1}), for any $b>0$,
\begin{equation}\label{L Phi 1}
\fint_X\Phi_1\left(\frac{\widehat{E}^n F}{b}\right)\omega_0^n\le L\fint_X\left(\Phi_1(\widehat{E}^n)\cdot \frac{F}{b}+\widehat{E}^n\cdot\Phi_1\left(\frac{F}{b}\right)\right)\omega_0^n.
\end{equation}

We estimate two terms on by one. First of all, by monotonicity of $\Phi_1(t)/t$ and (\ref{growth condition for Phi 3}) and (\ref{integrability comparison 4.5}), we have
\begin{eqnarray*}
\frac{\Phi_1(\widehat{E}^n)}{\widehat{E}^n} &\le& \frac{\Phi_1\left(1+\frac{E(\widehat{\mathcal{G}})^n}{J}\right)}
{1+\frac{E(\widehat{\mathcal{G}})^n}{J}}\\
&\le& \frac{\Phi_1(2)}{2}+\frac{\Phi_1\left(2\frac{E(\widehat{\mathcal{G}})^n}{J}\right)}{2\frac{E(\widehat{\mathcal{G}})^n}{J}}\\
&\le& \frac{\Phi_1(2)}{2}+L^2\left(\frac{\Phi_1\left(\frac{2}{J}\right)}{\frac{2}{J}}
+\frac{\Phi_1\left(E(\widehat{\mathcal{G}})^n\right)}{E(\widehat{\mathcal{G}})^n}\right)\\
&\le& C\left(1+\frac{\widehat{\mathcal{G}}}{E(\widehat{\mathcal{G}})^{n-1}}\right),
\end{eqnarray*}
where $C$ is a constant depending on the functions $\Phi_1$ and the values $E(1)$ and $L$. By the definition of $\widehat{E}$,
\begin{eqnarray*}
\Phi_1(\widehat{E}^n) &\le& C\left(1+\frac{E(\widehat{\mathcal{G}})^n}{J}\right)\cdot\left(1+\frac{\widehat{\mathcal{G}}}{E(\widehat{\mathcal{G}})^{n-1}}\right) \\
&\le& C\left(1+\frac{E(\widehat{\mathcal{G}})\cdot\widehat{\mathcal{G}}}{J\Phi_1(1)}+\frac{\widehat{\mathcal{G}}}{E(1)^{n-1}}
+\frac{E(\widehat{\mathcal{G}})\cdot\widehat{\mathcal{G}}}{J}\right)\\
&\le&C\cdot\left(1+\widehat{\mathcal{G}}+\frac{E\left(\widehat{\mathcal{G}}\right)\cdot \widehat{\mathcal{G}}}{J}\right).
\end{eqnarray*}
Thus,
\begin{equation}\label{L Phi 2}
\fint_X \Phi_1(\widehat{E}^n)\cdot \frac{F}{b}\omega_0^n\le \frac{C}{b},
\end{equation}
for some $C$ depending on the function $\Phi_1$ and the values $E(1)$ and $L$.

As for the second integrand, we have, by (\ref{integrability comparison 1}) (\ref{integrability comparison 0}) (\ref{integrability comparison 00}) and the lower bound of $J$,
\begin{eqnarray*}
\widehat{E}^n\cdot\Phi_1\left(\frac{F}{b}\right)&=&\left(1+\frac{E(\widehat{\mathcal{G}})^n}{J}\right)\cdot\Phi_1\left(\frac{F}{b}\right)\\
&\le& \Phi_1\left(\frac{F}{b}\right)+\frac{1}{J}\cdot\left(\Phi_2^*\left(E(\widehat{\mathcal{G}})^n\right)+\frac{\Phi(F/b)}{F/b}\right)\cdot \frac{F}{b}\\
&\le& \Phi_1\left(\frac{F}{b}\right)+\frac{E(\widehat{\mathcal{G}})\cdot\widehat{\mathcal{G}}}{J}\cdot \frac{F}{b}+\frac{\Phi\left(\frac{F}{b}\right)}{E(1)}.
\end{eqnarray*}
By the comparison (\ref{growth condition for Phi 1}) we have
$$\Phi_1(t)\le C\left(1+\Phi(t)\right),\qquad\forall ~t\ge 0,$$
for some constant $C>0$. So,
\begin{equation*}\label{L Phi 3}
\fint_X \widehat{E}^n\cdot\Phi_1\left(\frac{F}{b}\right)\omega_0^n
\le \fint_X\left(C\left(1+\Phi\left(\frac{F}{b}\right)\right)+\frac{E(\widehat{\mathcal{G}})\cdot\widehat{\mathcal{G}}}{J}\cdot \frac{F}{b}
+\frac{\Phi\left(\frac{F}{b}\right)}{E(1)}\right)\omega_0^n.
\end{equation*}
Taking $b=1+\|F\|_{L^\Phi}$ we get
\begin{equation}\label{L Phi 3b}
\fint_X \widehat{E}^n\cdot\Phi_1\left(\frac{F}{b}\right)\omega_0^n \le 2C+1+\frac{1}{E(1)}.
\end{equation}
Combining with (\ref{L Phi 2}) gives
\begin{equation*}
\fint_X\Phi_1\left(\frac{\widehat{E}^n F}{b}\right)\omega_0^n \le C,
\end{equation*}
for some $C$ depending on the functions $\Phi_1,\Phi$ and the values $E(1)$ and $L$. In particular
\begin{equation}
\|\widehat{E}^n F\|_{L^{\Phi_1}}\le C\cdot (1+\|F\|_{L^\Phi}).
\end{equation}

By the $L^\infty$ estimate to CMA, Theorem \ref{infinity estimate: 2},
\begin{equation}
\|\psi-\mathcal{V}_\theta\|\le C(\omega_0,\Phi_1;E,A,L,K,n).
\end{equation}
Following a similar argument as \cite[Corollary 4.1]{GuPhSoSt24-1}, we have
\begin{equation}\label{difference of solution bd}
\sup_X|\psi-\varphi_0|\le C.
\end{equation}

Then consider the following linear equation on $X$
\begin{equation}\label{Laplace-u-k}
\Delta_{\omega}u=-\widehat{E}+\fint_X \widehat{E}\omega^n,\quad \fint_Xu\omega^n=0.
\end{equation}
Let
\begin{equation}\label{v-k}
v=(\psi-\varphi_0)-\fint_X(\psi-\varphi_0)\omega^n+B^{-\frac{1}{n}}u.
\end{equation}
Then by (\ref{v-k}) and (\ref{cma-psi-k}), we have
$$\Delta_{\omega}v\ge-n.$$
Applying the Green's formula, we have
$$v(x)=\fint_Xv\omega^n+\int_XG(x,\cdot)(-\Delta_{\omega}v)\omega^n\le C_1,\qquad \forall ~~x\in X.$$
Together with the $L^\infty$ estimate of $\varphi_0$ and $\psi$, say (\ref{difference of solution bd}), we obtain
$$\sup u\le C_2+C_1\cdot B^{1/n}.$$
Applying the Green's formula again, we have
$$u(x)=\int_X\mathcal{G}(x,\cdot)\left(\widehat{E}-\fint_X \widehat{E}\omega^n\right)\omega^n.$$
Hence,
\begin{eqnarray*}
\int_X\mathcal{G}(x,\cdot)\cdot \widehat{E}\omega^n
\le u(x)+C_3\fint_X \widehat{E}\omega^n
\le C_4+C_3\fint_X \left(1+\frac{E(\widehat{\mathcal{G}})}{J^{1/n}}\right)\omega^n.
\end{eqnarray*}
Using that $E\big(\widehat{\mathcal{G}}\big)\le\widehat{\mathcal{G}}/\Phi(1)$ we get
\begin{equation*}
\int_X\mathcal{G}(x,\cdot)\cdot \widehat{E}\omega^n \le C_4+C_3\left(1+\frac{1}{\Phi(1)E(1)^{1/n}}\right) \le C_5.
\end{equation*}
Taking $x=x_0$ we have
\begin{equation*}
J^{\frac{n-1}{n}}=\int_X\mathcal{G}\cdot\frac{E\big(\widehat{\mathcal{G}}\big)}{J^{1/n}}\omega^n \le \int_X\mathcal{G}\cdot \widehat{E}\omega^n \le C_5.
\end{equation*}
Here the constants $C_i=C_i(\omega_0,\Phi,\Phi_1,E;A,L,K,n)$ for each $i=1,2\cdots$. It gives a required bound of $J$.
\end{proof}

\begin{lemma}\label{Gradient-Green}
Suppose that the function $E$ satisfies a further integrability condition
\begin{equation}\label{integrability comparison 5}
\int_1^\infty\frac{dt}{tE(t)}<\infty.
\end{equation}
There exists $C=C(\omega_0,\Phi,\Phi_1,E; A,K,A_1,K,n)$, independent of $x_0$, such that
\begin{equation}\label{L1-gradient-Green}
\int_X|\nabla \mathcal{G}|\omega^n\le C.
\end{equation}
\end{lemma}
\begin{proof}
Define the function
\begin{equation}
\xi(t)=\int_t^\infty\frac{ds}{sE(s)},\quad t\in[1,\infty).
\end{equation}
We extend the definition $\xi(t)=\xi(1)$ for $0\le t\le 1$. By assumption (\ref{integrability comparison 5}), $\xi$ is a bounded and decreasing function so that $\xi(\infty)=0$.

Let $u=\xi\circ\widehat{\mathcal{G}}$. Applying the Green's formula to $u$, we have
\begin{eqnarray*}
0=u(x_0)&=&\fint_X u\omega^n+\int_X G(x_0,y)(-\Delta_{\omega}u)(y)\omega^n(y)\\
&=&\fint_X u\omega^n+\int_X|\nabla_y\mathcal{G}(x_0,y)|^2_{\omega(y)}\cdot\xi'\left(\widehat{\mathcal{G}}(x_0,y)\right)\cdot[\omega^n]\cdot\omega^n.
\end{eqnarray*}
By Cauchy-Schwarz inequality, and applying the previous lemma, we get
\begin{eqnarray*}
\int_X|\nabla \mathcal{G}|\omega^n&\le&
-\int_X|\nabla\mathcal{G}|^2\cdot\xi'\big(\widehat{\mathcal{G}}\big)\cdot[\omega^n]\cdot\omega^n
-\int_X \frac{1}{\xi'\big(\widehat{\mathcal{G}}\big)\cdot[\omega^n]}\omega^n\\
&=&\fint_X u\omega^n+\int_X \frac{\widehat{\mathcal{G}}}{[\omega^n]}\cdot E(\widehat{\mathcal{G}})\omega^n\\
&\le&\xi(1)+C,
\end{eqnarray*}
where we used the previous lemma in the last inequality. The proof is complete.
\end{proof}

\begin{theorem}[Main Theorem]\label{main theorem}
Let $(X,\omega_0)$ be a K\"ahler manifold. Let $\varphi_0$ be a solution to CMA {\rm(\ref{MA: smooth case})} such that $\omega=\theta+\sqrt{-1}\partial\bar{\partial}\varphi_0$ is a K\"ahler metric. Assume that $F\in L^\Phi$ for some $N$-function $\Phi$. If there exists $\Phi_1$ satisfying conditions {\rm(C1)-(C3)} and the associated $E$ satisfies the integrability condition {\rm(\ref{integrability comparison 5})}, then
\begin{equation}
\diam(X,\omega)\le C(\omega_0,\Phi,\Phi_1;A,K,L,n).
\end{equation}
\end{theorem}
\begin{proof}
Let $(X,\omega)$ be the K\"ahler metric, $\omega=\theta+\sqrt{-1}\partial\bar{\partial}\varphi$ where $\varphi$ is a solution to CMA. Fix any base point $x_0$. Consider the distance function $d_{\omega}(x_0,x)$. By Green's formula
$$d_{\omega}(x_0,x)=\fint_X d_{\omega}(x_0,y)\omega^n(y)+\int_X\langle \nabla_y\mathcal{G}(x,y),\nabla_yd_{\omega}(x_0,y)\rangle_{\omega(y)}\omega^n(y).$$
At $x=x_0,$ we have $d_{\omega}(x_0,x_0)=0$ and
\begin{eqnarray*}
\fint_X d_{\omega}(x_0,y)\omega^n(y)&=&-\int_X\langle \nabla_y\mathcal{G}(x_0,y),\nabla_yd_{\omega}(x_0,y)_{\omega(y)}\rangle\omega^n(y)
\le\int_X|\nabla_y\mathcal{G}(x_0,y)|\omega^n(y).
\end{eqnarray*}
By gradient estimate for Green's function (\ref{L1-gradient-Green}), for any $x\in X$,
$$d_{\omega}(x_0,x)~\le \int_X|\nabla_y\mathcal{G}(x_0,y)|\omega^n(y)+\int_X|\nabla_y\mathcal{G}(x,y)|\omega^n(y)\le C.$$
The proof is complete.
\end{proof}

\subsection{Local integral bound of Green function and volume of geodesic balls}\label{local volume of geodesic balls}

Let $B_r(x_0)$ be a geodesic ball in $(X,\omega)$. Choose $\eta\in C^\infty(B_r(x_0))$ such that $\eta\ge 0$ and
$$\eta=1\mbox{ in }B_{r/2}(x_0),\qquad |\nabla\eta|\le \frac{4}{r}.$$
For any $z_0\in X\backslash B_r(x_0)$,
$$0=\eta(z_0)=\fint_X\eta\omega^n+\int_X\langle\nabla\eta(x),\nabla\mathcal{G}(z_0,x)\rangle\omega^n(x).$$
On the other hand,
$$1=\eta(x_0)=\fint_X\eta\omega^n+\int_X\langle\nabla\eta(x),\nabla\mathcal{G}(x_0,x)\rangle\omega^n.$$
The two formulas gives
$$1\le\int_X\langle\nabla\eta(x),\nabla\mathcal{G}(x_0,x)\rangle\omega^n-\int_X\langle\nabla\eta(x),\nabla\mathcal{G}(z_0,x)\rangle\omega^n(x)
\le\sup_{z\in X}\frac{8}{r}\int_{B_r(x_0)}|\nabla\mathcal{G}(z,\cdot)|\omega^n.$$
In particular
\begin{equation}\label{Gradient-Green-10}
\frac{r}{8}\le\sup_{z\in X}\int_{B_r(x_0)}|\nabla\mathcal{G}(z,\cdot)|\omega^n.
\end{equation}
In the following we show how to get a uniform control of $\int_{B_r(x_0)} |\nabla\mathcal{G}(z,\cdot)|\omega^n$.

We start with a local estimate of Green function. Extend the definition of $E$ so that $E(t)=E(1)$ when $0\le t\le 1$. Define
\begin{equation}
\Psi(t)=t\cdot E(t),\qquad\forall ~~t\ge 0.
\end{equation}
It is increasing strictly in $t$ and satisfies that $\lim_{t\rightarrow\infty}\Psi(t)/t=\infty$.

\begin{lemma}\label{local Green function estimate}
Assume as before. For any $r>0$ and $x_0\in X$, we have
\begin{equation}
\int_{B_r(x_0)}\mathcal{G}\omega^n\le\frac{C}{(\Psi^*)^{-1}\left(\frac{[\omega^n]}{\vol(B_r(x_0))}\right)},
\end{equation}
where $C=C(\omega_0,\Phi,\Phi_1,E,A,K,L,n)$.
\end{lemma}
\begin{proof}
For any $\epsilon>0$,
\begin{eqnarray*}
\int_{B_r(x_0)}\mathcal{G}\omega^n\le\frac{\epsilon}{[\omega^n]}\cdot\int_{B_r(x_0)}\left(\Psi\left(\widehat{\mathcal{G}}\right)
+\Psi^*\left(\epsilon^{-1}\right)\right)\omega^n
\le\epsilon\left(C+\Psi^*\left(\epsilon^{-1}\right)\cdot\frac{\vol(B_r(x_0))}{[\omega^n]}\right).
\end{eqnarray*}
Choosing $\epsilon>0$ so that $\Psi^*\left(\epsilon^{-1}\right)=\vol(B_r(x_0))^{-1}[\omega^n]$ we get the required estimate.
\end{proof}

Take a smaller increasing positive function $\widetilde{E}(t)$ such that
\begin{equation}\label{choice of-widetildeE}
\widetilde{E}(t)\le E(t),\qquad\lim_{t\rightarrow\infty}\frac{\widetilde{E}(t)}{E(t)}=0,\qquad\int_1^\infty\frac{dt}{t\widetilde{E}(t)}<\infty.
\end{equation}
It is obvious that $\widetilde{E}$ satisfies (\ref{integrability comparison 0}) and (\ref{integrability comparison 1}). The existence of such $\widetilde{E}$ is trivial. We may also assume that
$$\int_1^\infty\frac{dt}{t\widetilde{E}(t)}\le 2\cdot\int_1^\infty\frac{dt}{tE(t)}.$$
Let $\widetilde{\Psi}$ be an increasing function defined on $[0,\infty)$ so that $\widetilde{\Psi}(0)=0$ and
\begin{equation}\label{choice of-widetildePsi}
\widetilde{\Psi}\circ\widetilde{E}=E.
\end{equation}

We refine the estimate in Lemma \ref{Gradient-Green}. Let $\mathcal{G}=\mathcal{G}(z,\cdot)$ for a fixed $z\in X$ and put $\widehat{\mathcal{G}}=[\omega^n]\cdot\mathcal{G}$. Define $\tilde{\xi}(t)=\int_t^\infty\frac{ds}{s\widetilde{E}(s)}$. Then, as in the proof of Lemma \ref{Gradient-Green},
\begin{eqnarray}
\int_{B_r(x_0)}|\nabla\mathcal{G}|\omega^n&\le&\left(-\int_{B_r(x_0)}|\nabla\mathcal{G}|^2
\tilde{\xi}'\left(\widehat{\mathcal{G}}\right)[\omega^n]\omega^n\right)^{1/2}
\left(-\int_{B_r(x_0)} \frac{\omega^n}{\tilde{\xi}'\left(\widehat{\mathcal{G}}\right)[\omega^n]}\right)^{1/2}\nonumber\\
&\le&\tilde{\xi}(1)^{1/2}\cdot\left(\int_{B_r(x_0)} \mathcal{G}\cdot \widetilde{E}\left(\widehat{\mathcal{G}}\right)\omega^n\right)^{1/2}.\label{Gradient-Green-11}
\end{eqnarray}
Using the complementary function $\Psi^*$ we get, for any $\epsilon>0$,
\begin{eqnarray*}
\int_{B_r(x_0)}\mathcal{G}\cdot\widetilde{E}\left(\widehat{\mathcal{G}}\right)\omega^n
&\le&\epsilon\cdot\int_{B_r(x_0)}\mathcal{G}\cdot\left(\widetilde{\Psi}\left(\widetilde{E}\circ\widehat{\mathcal{G}}\right)
+\widetilde{\Psi}^*\left(\epsilon^{-1}\right)\right)\omega^n\\
&\le&\epsilon\cdot\left(\int_X\mathcal{G}\cdot E\left(\widehat{\mathcal{G}}\right)\omega^n+\widetilde{\Psi}^*\left(\epsilon^{-1}\right)\cdot\int_{\vol(B_r(x_0))}\mathcal{G}\omega^n\right)\\
&\le&\epsilon\cdot\left(C+\frac{C\cdot\widetilde{\Psi}^*\left(\epsilon^{-1}\right)}{(\Psi^*)^{-1}\left(\frac{[\omega^n]}{\vol(B_r(x_0))}\right)}\right).
\end{eqnarray*}
Choosing $\epsilon>0$ so that
$$\widetilde{\Psi}^*\left(\epsilon^{-1}\right)=(\Psi^*)^{-1}\left(\frac{[\omega^n]}{\vol(B_r(x_0))}\right),$$
we get
\begin{equation}
\int_{B_r(x_0)}\mathcal{G}\cdot\widetilde{E}\left(\widehat{\mathcal{G}}\right)\omega^n\le \frac{C}{(\Psi^*\circ\widetilde{\Psi}^*)^{-1}\left(\frac{[\omega^n]}{\vol(B_r(x_0))}\right)}.
\end{equation}
Substituting into (\ref{Gradient-Green-11}) and (\ref{Gradient-Green-10}) we finally get
$$\frac{r^2}{64}\le\frac{C}{(\Psi^*\circ\widetilde{\Psi}^*)^{-1}\left(\frac{[\omega^n]}{\vol(B_r(x_0))}\right)}.$$

We conclude with the following proposition.

\begin{proposition}\label{Volume of geodesic balls}
Assume as above. For any $x_0\in X$, the volume of geodesic balls satisfies the lower bound
\begin{equation}\label{volume of geodesic balls}
\frac{\vol(B_r(x_0))}{[\omega^n]}\ge\frac{1}{\Psi^*\circ\widetilde{\Psi}^*\left(\frac{C}{r^2}\right)},\quad \forall ~~r<\frac{1}{2}\diam(X,\omega),
\end{equation}
for some $C=C(\omega_0,\Phi,\Phi_1,E;A,K,L,n)$.
\end{proposition}

\subsection{Slow growth $N$-functions}\label{SG-N-function-diam-bound}

We check the assumptions in the previous subsection for slow growth $N$-functions. Let $\Phi$ be an $N$-function such that
\begin{equation}
\Phi(t)\ge g_0(t)\cdot g_1(t)^{p_1}\cdots g_k(t)^{p_k},
\end{equation}
for any sufficiently large $t$, where $k\ge 1$. The function $g_0(t)=t$ and $g_{i+1}(t)=\log(1+g_i(t))$ inductively for $i\ge 0$. Assume that each $p_i\ge n$ and $p_k>n$ so that $\Phi$ satisfies the first integrability condition (\ref{integrability 0}). See the discuss in Section 5.

Let $\Phi_1$ be a smaller $N$-function of the form
\begin{equation}
\Phi_1=g_0\cdot g_1^n\cdots g_\ell^n\cdot g_{\ell+1}^{q},
\end{equation}
for some $q>n$. In application we shall take a larger $\ell$, say $\ell\ge k+2$.

\begin{lemma}
There is $L>0$ so that both $\Phi$ and $\Phi_1$ satisfy the growth condition {\rm(\ref{growth condition for Phi 3})}.
\end{lemma}
\begin{proof}
We only check (\ref{growth condition for N-function: 1}) for $\Phi_1$. It suffices to check that for large $t_1,t_2$. By symmetry we may assume that $t_1\ge t_2$ are large numbers, then
$$g_1(t_1t_2)\le 2\log (t_1t_2)\le 4\log t_1\le 4g_1(t_1),$$
and by induction one can show that $g_i(t_1t_2)\le 4g_i(t_1)$ for any $1\le i\le \ell+1$. Thus,
$$\Phi_1(t_1t_2)\le 4^{q_1+\cdots+ q_{\ell+1}}t_1t_2 g_1(t_1)^{q_1}\cdots g_\ell(t_1)^{q_\ell}g_{\ell+1}(t_1)^{q_{\ell+1}}=4^{q_1+\cdots+ q_{\ell+1}}\Phi_1(t_1)t_2.$$
So one can choose $L=4^{q_1+\cdots+ q_{\ell+1}}$ for large $t_1$ and $t_2$.
\end{proof}

Let $\Phi_2$ be the function determined by (\ref{integrability comparison 1}), namely,
$$\Phi_2\left(\frac{\Phi_1(t)}{t}\right)=\frac{\Phi(t)}{t}.$$
It can be solved out that
\begin{equation}
\Phi_2(t) \sim g_0(t)^{\frac{p_1}{n}}g_1(t)^{p_2-p_1}\cdots g_{k-1}(t)^{p_k-p_1}g_k(t)^{-n}\cdots g_{\ell}(t)^{-\frac{qp_1}{n}}.
\end{equation}
See Section 5 for the meaning of the symbol $\sim$.

\subsubsection{$k=1$ and $p>n$.}

In this case $\Phi=g_0g_1^p$, $\ell\ge 2$ and
\begin{equation*}
\Phi_2(t) \sim t^{\frac{p}{n}}g_1(t)^{-p}\cdots g_{\ell-1}(t)^{-p}g_\ell(t)^{-\frac{pq}{n}},
\end{equation*}
and
\begin{equation*}
\Phi_2^*(t) \sim t^{\frac{p}{p-n}}g_1(t)^{\frac{pn}{p-n}}\cdots g_{\ell-1}(t)^{\frac{pn}{p-n}}g_\ell(t)^{\frac{pq}{p-n}}.
\end{equation*}
The $E_1$ and $E_2$ in (\ref{integrability comparison 2})-(\ref{integrability comparison 4}) satisfy
\begin{equation*}
E_1(t) \sim t^{\frac{1}{n-1}}g_1(t)^{-\frac{n}{n-1}}\cdot g_2(t)^{-\frac{n}{n-1}}\cdots g_{\ell}(t)^{-\frac{n}{n-1}}\cdot g_{\ell+1}(t)^{-\frac{q}{n-1}},
\end{equation*}
and
\begin{equation*}
E_2(t) \sim t^{\frac{p-n}{pn-p+n}}g_1(t)^{-\frac{pn}{pn-p+n}}\cdots g_{\ell-1}(t)^{-\frac{pn}{pn-p+n}}g_\ell(t)^{-\frac{pq}{pn-p+n}}.
\end{equation*}
Noticing that the orders of $t$ in $E_1$ and $E_2$ satisfies $\frac{1}{n-1}>\frac{p-n}{pn-p+n}$, we have
\begin{equation}
E(t)=\min(E_1(t), E_2(t)) \sim E_2(t).
\end{equation}
The integrability condition (\ref{integrability comparison 5}) is fulfilled trivially. The refined Green function estimate (\ref{L1+-Green-function}) reads
\begin{equation}
\fint_X\widehat{\mathcal{G}}^{\frac{pn}{pn-p+n}}\cdot g_1\left(\widehat{\mathcal{G}}\right)^{-\frac{pn}{pn-p+n}}\cdots g_{\ell-1}\left(\widehat{\mathcal{G}}\right)^{-\frac{pn}{pn-p+n}}\cdot g_\ell\left(\widehat{\mathcal{G}}\right)^{-\frac{pq}{pn-p+n}} \omega^n\le C.
\end{equation}
It is a precise refinement of the integral Green estimate in \cite{GuPhSoSt24-1,GuPhSt24,Liu24+,ZwZy25+} in Orlicz space $L^\Phi$ with $\Phi=g_0g_1^p$, $p>n$.

The integrability condition (\ref{integrability comparison 5}) is trivial in this case. We next check the volume non-collapsing estimate.

First of all,
\begin{equation*}
\Psi(t)=tE(t)\sim t^{\frac{pn}{pn-p+n}}g_1(t)^{-\frac{pn}{pn-p+n}}\cdots g_{\ell-1}(t)^{-\frac{pn}{pn-p+n}}g_\ell(t)^{-\frac{pq}{pn-p+n}},
\end{equation*}
and
\begin{equation*}
\Psi^*(t) \sim t^{\frac{pn}{p-n}}g_1(t)^{\frac{pn}{p-n}}\cdots g_{\ell-1}(t)^{\frac{pn}{p-n}}g_\ell(t)^{\frac{pq}{p-n}}.
\end{equation*}
Then, taking $\widetilde{E}=g_1\cdots g_\ell g_{\ell+1}^{q'}$ for some $q'>2$ so that $\int\frac{dt}{t\widetilde{E}(t)}$ is integrable. Then,
\begin{equation*}
\widetilde{\Psi}(t)=E\circ\widetilde{E}^{-1}(t) \sim \exp\left\{tg_1(t)^{-1}\cdots g_{\ell-1}^{-1}g_\ell^{-q'}\right\},
\end{equation*}
and
\begin{equation*}
\widetilde{\Psi}^*(t) \sim t\widetilde{E}(t)=tg_1(t)\cdots g_\ell(t)g_{\ell+1}(t)^{q'}.
\end{equation*}
The composition
\begin{equation*}
\Psi^*\circ\widetilde{\Psi}^*(t) \sim t^{\frac{pn}{p-n}}g_1(t)^{\frac{2pn}{p-n}}\cdots g_{\ell-1}(t)^{\frac{2pn}{p-n}}g_{\ell}(t)^{\frac{p(n+q)}{p-n}}g_{\ell+1}(t)^{\frac{pqq'}{p-n}}.
\end{equation*}
So the volume non-collapsing estimate in (\ref{volume of geodesic balls}) reads
\begin{equation}\label{almost-sharp-volume-estimate-1}
\frac{\vol(B_r(x_0))}{[\omega^n]}\ge \frac{C\cdot r^{\frac{2pn}{p-n}}}{g_1(\frac{1}{r^2})^{\frac{2pn}{p-n}}\cdots g_{\ell-1}(\frac{1}{r^2})^{\frac{2pn}{p-n}}g_{\ell}(\frac{1}{r^2})^{\frac{p(n+q)}{p-n}}g_{\ell+1}(\frac{1}{r^2})^{\frac{pqq'}{p-n}}}.
\end{equation}
where $C=C(\omega_0,\Phi,\Phi_1,E; A,K,L,K,n)$.
This is exactly the volume estimate proved in \cite{ZwZy25+} (see also \cite{GuPhSt24,GuPhSoSt24-1,GuPhSoSt23,Liu24+}).

\subsubsection{$k\ge 2$ and $p_1=n$ and $p_i\ge n$ for each $i$ and $p_{i_0}>n$ for some $i_0$.}

Let $2\le j_0\le k$ be the first $j$ such that $p_{j_0}>n$. Then we have
$$\Phi_2(t) \sim g_0(t)g_{j_0}(t)^{p_{j_0}-n}\cdots g_{k-1}(t)^{p_k-n}g_k(t)^{-n}\cdots g_{\ell}(t)^{-q}.$$

{\bf Subcase 2.1: $k=2$ and $p_1=n$, $p_2=p>n$.} In this case we have $\ell\ge 3$ and
$$\Phi_2(t) \sim t g_1(t)^{p-n}g_2(t)^{-n}\cdots g_{\ell-1}(t)^{-n}g_\ell(t)^{-q},$$
and
$$\Phi_2^*(t) \sim \exp\left\{t^{\frac{1}{p-n}}g_1(t)^{\frac{n}{p-n}}\cdots g_{\ell-2}(t)^{\frac{n}{p-n}}\cdot g_{\ell-1}(t)^{\frac{q}{p-n}}\right\}.$$
Then,
$$E_1(t)\sim t^{\frac{1}{n-1}}g_1(t)^{-\frac{n}{n-1}}\cdot g_2(t)^{-\frac{n}{n-1}}\cdots g_{\ell}(t)^{-\frac{n}{n-1}}\cdot g_{\ell+1}(t)^{-\frac{q}{n-1}},$$
and,
$$E_2(t) \sim g_1(t)^{\frac{p-n}{n}}\cdot g_2(t)^{-1}\cdots g_{\ell-1}(t)^{-1}\cdot g_\ell(t)^{-\frac{q}{n}}.$$
So,
\begin{equation}
E(t)\sim E_2(t) \sim g_1(t)^{\frac{p-n}{n}}\cdot g_2(t)^{-1}\cdots g_{\ell-1}(t)^{-1}\cdot g_\ell(t)^{-\frac{q}{n}}.
\end{equation}
When $\Phi=g_0g_1^ng_2^p$ for some $p>n$, the Green function estimate (\ref{L1+-Green-function}) reads
\begin{equation}
\fint_X\widehat{\mathcal{G}}\cdot g_1\left(\widehat{\mathcal{G}}\right)^{\frac{p-n}{n}}\cdot g_2\left(\widehat{\mathcal{G}}\right)^{-1}\cdots g_{\ell-1}\left(\widehat{\mathcal{G}}\right)^{-1}\cdot g_\ell\left(\widehat{\mathcal{G}}\right)^{-\frac{q}{n}} \omega^n\le C.
\end{equation}

The integrability condition (\ref{integrability comparison 5}) remains true in this case if $p>2n$.

We next check the volume non-collapsing estimate. By definition,
\begin{equation*}
\Psi(t)=tE(t)\sim tg_1(t)^{\frac{p-n}{n}}\cdot g_2(t)^{-1}\cdots g_{\ell-1}(t)^{-1}\cdot g_\ell(t)^{-\frac{q}{n}},
\end{equation*}
and
\begin{equation*}
\Psi^*(t) \sim \exp\left\{t^{\frac{n}{p-n}}g_1(t)^{\frac{n}{p-n}}\cdots g_{\ell-2}(t)^{\frac{n}{p-n}}g_{\ell-1}(t)^{\frac{q}{p-n}}\right\}.
\end{equation*}
Then, taking $\widetilde{E}=g_1\cdots g_{\ell} g_{\ell+1}^{q'}$ for some $q'>1$ so that $\int\frac{dt}{t\widetilde{E}(t)}$ is integrable. Assume $p>2n$, then,
\begin{equation*}
\widetilde{\Psi}(t)=E\circ\widetilde{E}^{-1}(t) \sim t^{\frac{p-n}{n}}g_1(t)^{-\frac{p}{n}}\cdots g_{\ell-2}(t)^{-\frac{p}{n}}g_{\ell-1}(t)^{-\frac{p+q-n}{n}}g_\ell(t)^{-\frac{(p-n)q'}{n}},
\end{equation*}
and
\begin{equation*}
\widetilde{\Psi}^*(t) \sim t^{\frac{p-n}{p-2n}}g_1(t)^{\frac{p}{p-2n}}\cdots g_{\ell-2}(t)^{\frac{p}{p-2n}}g_{\ell-1}(t)^{\frac{p+q-n}{p-2n}}g_\ell(t)^{\frac{(p-n)q'}{p-2n}}.
\end{equation*}
The composition
\begin{equation}\label{almost-sharp-volume-estimate: 0}
\Psi^*\circ\widetilde{\Psi}^*(t) \sim \exp\left\{t^{\frac{n}{p-2n}}g_1(t)^{\frac{2n}{p-2n}}\cdots g_{\ell-2}(t)^{\frac{2n}{p-2n}}g_{\ell-1}(t)^{\frac{n+q}{p-2n}}g_\ell(t)^{\frac{nq'}{p-2n}}\right\}.
\end{equation}
Using (\ref{volume of geodesic balls}) gives the volume non-collapsing estimate.

{\bf Subcase 2.2: $k\ge 3$ and $p_1=\cdots =p_{k-1}=n$, $p_k=p>n$.} We have
\begin{equation*}
\Phi_2(t) \sim t\cdot g_{k-1}(t)^{p-n}\cdot g_k(t)^{-n}\cdots g_{\ell-1}(t)^{-n}\cdot g_\ell(t)^{-q},
\end{equation*}
and
\begin{equation*}
\Phi_2^*(t) \sim G_{k-1}\left(t^{\frac{1}{p-n}}\cdot g_1(t)^{\frac{n}{p-n}}\cdots g_{\ell-k}(t)^{\frac{n}{p-n}}\cdot g_{\ell-k+1}(t)^{\frac{q}{p-n}}\right),
\end{equation*}
where $G_{k-1}$ is the inverse function of $g_{k-1}$. The $E_1$ and $E_2$ in (\ref{integrability comparison 2})-(\ref{integrability comparison 4}) satisfy
\begin{equation*}
E_1(t) \sim t^{\frac{1}{n-1}}g_1(t)^{-\frac{n}{n-1}}\cdots g_{k-1}(t)^{-\frac{n}{n-1}}\cdot g_k(t)^{-\frac{q}{n-1}},
\end{equation*}
and
\begin{equation*}
E_2(t) \sim  g_{k-1}(t)^{\frac{p-n}{n}}\cdot g_k(t)^{-1}\cdots g_{k+\ell-3}(t)^{-1}\cdot g_{k+\ell-2}(t)^{-\frac{q}{n}}.
\end{equation*}
Similarly,
\begin{equation}
E(t)\sim E_2(t) \sim  g_{k-1}(t)^{\frac{p-n}{n}}\cdot g_k(t)^{-1}\cdots g_{k+\ell-3}(t)^{-1}\cdot g_{k+\ell-2}(t)^{-\frac{q}{n}}.
\end{equation}
When $\Phi=g_0g_1^n\cdots g_{k-1}^ng_2^p$ for some $p>n$, the Green function estimate (\ref{L1+-Green-function}) reads
\begin{equation}
\fint_X\widehat{\mathcal{G}}\cdot g_{k-1}\left(\widehat{\mathcal{G}}\right)^{\frac{p-n}{n}}\cdot g_k\left(\widehat{\mathcal{G}}\right)^{-1}\cdots g_{k+\ell-3}\left(\widehat{\mathcal{G}}\right)^{-1}\cdot g_{k+\ell-2}\left(\widehat{\mathcal{G}}\right)^{-\frac{q}{n}} \omega^n\le C.
\end{equation}
The formula coincides with that in the previous case. However, the integrability condition (\ref{integrability comparison 5}) fails in this case.

{\bf Subcase 2.3: $k\ge 3$ and $p_1=\cdots =p_{j_0-1}=n$, $p_{j_0}=p>n$ for some $2\le j_0<k$.} We have
\begin{equation*}
\Phi_2(t) \sim t\cdot g_{j_0-1}(t)^{p_{j_0}-n}\cdots g_{k-1}(t)^{p_k-n}\cdot g_k(t)^{-n}\cdots g_{\ell-1}(t)^{-n}\cdot g_\ell(t)^{-q},
\end{equation*}
and
\begin{equation*}
\Phi_2^*(t) \sim G_{j_0-1}\left(t^{\frac{1}{p_{j_0}-n}}\cdot g_1^{-\frac{p_{j_0+1}-n}{p_{j_0}-n}}\cdots g_{k-j_0}^{-\frac{p_k-n}{p_{j_0}-n}} \cdot g_{k-j_0+1}^{\frac{n}{p_{j_0}-n}}\cdots g_{\ell-j_0}^{\frac{n}{p_{j_0}-n}}\cdot g_{\ell-j_0+1}^{\frac{q}{p_{j_0}-n}}\right),
\end{equation*}
where $g_i=g_i(t)$. The $E_1$ and $E_2$ in (\ref{integrability comparison 2})-(\ref{integrability comparison 4}) satisfy
\begin{equation*}
E_1(t) \sim t^{\frac{1}{n-1}}g_1(t)^{-\frac{n}{n-1}}\cdots g_{k-1}(t)^{-\frac{n}{n-1}}\cdot g_k(t)^{-\frac{q}{n-1}},
\end{equation*}
and
\begin{equation*}
E_2(t) \sim g_{j_0-1}(t)^{\frac{p_{j_0}-n}{n}}\cdot g_{j_0}(t)^{\frac{p_{j_0+1}-n}{n}}\cdots g_{k-1}(t)^{\frac{p_k-n}{n}} \cdot g_{k}(t)^{-1}\cdots g_{\ell-1}(t)^{-1}\cdot g_{\ell}(t)^{-\frac{q}{n}}.
\end{equation*}
Similarly,
\begin{equation}
E(t)\sim E_2(t).
\end{equation}
When $\Phi=g_0g_1^n\cdots g_{j_0-1}^ng_{j_0}^{p_{j_0}}\cdots g_k^{p_k}$ for some $p_{j_0}>n$, $2\le j_0<k$, the Green function estimate (\ref{L1+-Green-function}) reads
\begin{equation}
\fint_X\widehat{\mathcal{G}}\cdot g_{j_0-1}\left(\widehat{\mathcal{G}}\right)^{\frac{p_{j_0}-n}{n}}\cdot g_{j_0}\left(\widehat{\mathcal{G}}\right)^{\frac{p_{j_0+1}-n}{n}}\cdots g_{k-1}\left(\widehat{\mathcal{G}}\right)^{\frac{p_k-n}{n}} \cdot g_{k}\left(\widehat{\mathcal{G}}\right)^{-1}\cdots g_{\ell-1}\left(\widehat{\mathcal{G}}\right)^{-1}\cdot g_{\ell}\left(\widehat{\mathcal{G}}\right)^{-\frac{q}{n}}\omega^n\le C.
\end{equation}

The integrability condition (\ref{integrability comparison 5}) holds in this case if and only if the first $p_i\neq 2n$ for $i\ge 2$ must be strictly bigger than $2n$, namely
\begin{equation}
p_2=\cdots p_{j_0-1}=2n, \quad p_{j_0}>2n,
\end{equation}
for some $2\le j_0\le k$.

We assume that $3\le j_0\le k$. Then,
$$\Phi(t)=tg_1(t)^ng_2(t)^{2n}\cdots g_{j_0-1}(t)^{2n}g_{j_0}(t)^{p_{j_0}}\cdots g_k(t)^{p_k},$$
and,
$$E(t)\sim g_1(t)\cdots g_{j_0-2}(t)g_{j_0-1}(t)^{\frac{p_{j_0}-n}{n}}\cdots g_{k-1}(t)^{\frac{p_k-n}{n}}g_k(t)^{-1}\cdots g_{\ell-1}(t)^{-1}g_{\ell}(t)^{-\frac{q}{n}},$$
\begin{equation*}
\Psi(t)=tE(t)\sim tg_1(t)\cdots g_{j_0-2}(t)g_{j_0-1}(t)^{\frac{p_{j_0}-n}{n}}\cdots g_{k-1}(t)^{\frac{p_k-n}{n}}g_k(t)^{-1}\cdots g_{\ell-1}(t)^{-1}g_{\ell}(t)^{-\frac{q}{n}},
\end{equation*}
and
\begin{equation*}
\Psi^*(t) \sim \exp\left\{tg_1(t)^{-1}\cdots g_{j_0-3}(t)^{-1}g_{j_0-2}(t)^{-\frac{p_{j_0}-n}{n}}\cdots g_{k-2}(t)^{-\frac{p_k-n}{n}}g_{k-1}(t)\cdots g_{\ell-2}(t)g_{\ell-1}(t)^{\frac{q}{n}}\right\}.
\end{equation*}

Take $\widetilde{E}=g_1\cdots g_\ell g_{\ell+1}^{q'}$ for some $q'>1$ so that $\int\frac{dt}{t\widetilde{E}(t)}$ is integrable. Then,
\begin{equation*}
\widetilde{\Psi}(t)=E\circ\widetilde{E}^{-1}(t) \sim  tg_{j_0-2}(t)^{\frac{p_{j_0}-2n}{n}}\cdots g_{k-2}(t)^{\frac{p_{k}-2n}{n}}g_{k-1}(t)^{-2}\cdots g_{\ell-2}(t)^{-2}g_{\ell-1}(t)^{-1-\frac{q}{n}}g_\ell(t)^{-q'},
\end{equation*}
and
\begin{eqnarray*}
\widetilde{\Psi}^*(t) & \sim & G_{j_0-2}\left(t^{\frac{n}{p_{j_0}-2n}}g_1(t)^{-\frac{p_{j_0+1}-2n}{p_{j_0}-2n}}\cdots g_{k-j_0}(t)^{-\frac{p_k-2n}{p_{j_0}-2n}} g_{k-j_0+1}(t)^{\frac{2n}{p_{j_0}-2n}}\cdots \right.\\
&&\hspace{2cm}\left.\cdots g_{\ell-j_0}(t)^{\frac{2n}{p_{j_0}-2n}}
g_{\ell-j_0+1}(t)^{\frac{n+q}{p_{j_0}-2n}}g_{\ell-j_0+2}(t)^{\frac{nq'}{p_{j_0}-2n}}\right).
\end{eqnarray*}
The composition $\Psi^*\circ\widetilde{\Psi}^*(t) \sim  \exp\circ \widetilde{\Psi}^*(t)$, so
\begin{eqnarray}
\nonumber\Psi^*\circ\widetilde{\Psi}^*(t) & \sim & G_{j_0-1}\left(t^{\frac{n}{p_{j_0}-2n}}g_1(t)^{-\frac{p_{j_0+1}-2n}{p_{j_0}-2n}}\cdots g_{k-j_0}(t)^{-\frac{p_k-2n}{p_{j_0}-2n}} g_{k-j_0+1}(t)^{\frac{2n}{p_{j_0}-2n}}\cdots \right.\\
&&\hspace{2cm}\left.\cdots g_{\ell-j_0}(t)^{\frac{2n}{p_{j_0}-2n}}
g_{\ell-j_0+1}(t)^{\frac{n+q}{p_{j_0}-2n}}g_{\ell-j_0+2}(t)^{\frac{nq'}{p_{j_0}-2n}}\right).\label{almost-sharp-volume-estimate: 00}
\end{eqnarray}

If, on the other side, $j_0=2<k$, one can show that $\Psi^*\circ\widetilde{\Psi}^*(t)$ has the same formula as in (\ref{almost-sharp-volume-estimate: 00}). Moreover, the formula (\ref{almost-sharp-volume-estimate: 0}) in Subcase 2.1 is exactly the formula (\ref{almost-sharp-volume-estimate: 00}) in Subcase 2.3 if we take $p_3=\cdots =p_k=0$.

So the volume non-collapsing estimate in (\ref{volume of geodesic balls}) reads
\begin{equation}\label{almost-sharp-volume-estimate-n}
\frac{\vol(B_r(x_0))}{[\omega^n]}\ge \frac{C^{-1}}{G_{k_0-1}\circ f(\frac{C}{r^2})},
\end{equation}
for some constant $C>0$, where $f$ is defined by the (\ref{almost-sharp-volume-estimate: 00}).

We summarize the discussion into the following corollary.

\begin{corollary}
Let $\Phi$ be an $N$-function such that
\begin{equation}
\Phi(t)\succ g_0(t)g_1(t)^ng_2(t)^{2n}\cdots g_{k-1}(t)^{2n}g_k(t)^p
\end{equation}
for some $k\ge 2$ and $p>2n$. Let $\varphi_0$ be a solution to CMA {\rm(\ref{MA: smooth case})} with $F\in L^\Phi$. If $\omega=\theta+\sqrt{-1}\partial\bar{\partial}\varphi_0$ is a K\"ahler metric, then
$$\diam(X,\omega)\le C,$$
and
\begin{equation}\label{almost-sharp-volume-estimate: 010}
\frac{\vol(B_r(x_0))}{[\omega^n]}\ge\frac{C^{-1}}{G_{k-1}\left[\left(\frac{C}{r}\right)^{\frac{2n}{p-2n}}g_{1}\left(\frac{C}{r}\right)^{\frac{2n}{p-2n}} g_2\left(\frac{C}{r}\right)^{\frac{2n}{p-2n}}
g_3\left(\frac{C}{r}\right)^{\frac{n+q}{p-2n}}g_4\left(\frac{C}{r}\right)^{\frac{nq'}{p-2n}}\right]},
\end{equation}
for some $C=C(\omega_0,\Phi;A,K,n)$, where $G_{k-1}$ is the inverse function of $g_{k-1}$.
\end{corollary}
\begin{proof}
Define the function $\Phi_1=g_0(t)\cdot g_1(t)^n\cdots g_{k+2}(t)^n\cdot g_{k+3}(t)^{2n}$ and associated $E$ as above too. Applying Main Theorem \ref{main theorem} then yields the required diameter estimate. Applying (\ref{almost-sharp-volume-estimate-n}) then yields the required volume estimate (\ref{almost-sharp-volume-estimate: 010}).
\end{proof}

\subsection{Polynomial growth $N$-functions}\label{Polynomial growth N-functions for diamter-volume}

Let $\Phi$ be an $N$-function such that
\begin{equation}
\Phi(t)\ge \frac{t^p}{p},
\end{equation}
for any sufficiently large $t$, where $p>1$. Let $\Phi_1=g_0g_1^n\cdots g_\ell^ng_{\ell+1}^q$ be a slow growth function for some $\ell\ge 0$ and $q>n$. By definition (\ref{integrability comparison 1}), we have
\begin{equation*}
\Phi_2(t) \sim \exp\left\{t^{\frac{1}{n}}g_1(t)^{-1}\cdots g_{\ell-1}(t)^{-1}g_\ell(t)^{-\frac{q}{n}}\right\}.
\end{equation*}
It follows that
\begin{equation*}
\Phi_2^*(t) \sim \Phi_1(t)=g_0(t)g_1(t)^n\cdots g_\ell(t)^ng_{\ell+1}(t)^q.
\end{equation*}
See Theorem 3.2 in \cite{KrRu} and the discussion in Section 5.1. By (\ref{integrability comparison 2}) and (\ref{integrability comparison 4}) we have that
\begin{equation*}
E_1(t) \sim E_2(t) \sim t^{\frac{1}{n-1}}g_1(t)^{-\frac{n}{n-1}}\cdots g_\ell(t)^{-\frac{n}{n-1}}g_{\ell+1}(t)^{-\frac{q}{n-1}}.
\end{equation*}
The integrability condition (\ref{integrability comparison 5}) is trivial. Then refined integral estimate of Green function reads
\begin{equation}
\fint_X\widehat{\mathcal{G}}^{\frac{n}{n-1}}\cdot g_1\left(\widehat{\mathcal{G}}\right)^{-\frac{n}{n-1}}\cdots g_\ell\left(\widehat{\mathcal{G}}\right)^{-\frac{n}{n-1}}g_{\ell+1}\left(\widehat{\mathcal{G}}\right)^{-\frac{q}{n-1}}~\omega^n\le C.
\end{equation}
It is a precise refinement of the integral Green estimate in \cite{GuPhSt24,GuPhSoSt23,GuPhSoSt24-1,Liu24+,GuTo25,ZwZy25+} in Orlicz space $L^\Phi$ with $\Phi=t^p/p$, $p>1$.

Then we discuss the functions $\Psi$ and $\widetilde{\Psi}$. First of all,
\begin{equation*}
\Psi(t)=tE(t) \sim t^{\frac{n}{n-1}}g_1(t)^{-\frac{n}{n-1}}\cdots g_\ell(t)^{-\frac{n}{n-1}}g_{\ell+1}(t)^{-\frac{q}{n-1}},
\end{equation*}
and
\begin{equation*}
\Psi^*(t) \sim t^ng_1(t)^n\cdots g_\ell(t)^ng_{\ell+1}(t)^q.
\end{equation*}
Then, taking $\widetilde{E}=g_1\cdots g_\ell g_{\ell+1}^{q'}$ for some $q'>2$ so that $\int\frac{dt}{t\widetilde{E}(t)}$ is integrable. Then,
\begin{equation*}
\widetilde{\Psi}(t)=E\circ\widetilde{E}^{-1}(t) \sim \exp\left\{tg_1(t)^{-1}\cdots g_{\ell-1}^{-1}g_\ell^{-q'}\right\},
\end{equation*}
and
\begin{equation*}
\widetilde{\Psi}^*(t) \sim t\widetilde{E}(t)=tg_1(t)\cdots g_\ell(t)g_{\ell+1}(t)^{q'}.
\end{equation*}
The composition
\begin{equation*}
\Psi^*\circ\widetilde{\Psi}^*(t) \sim t^ng_1(t)^{2n}\cdots g_\ell(t)^{2n}g_{\ell+1}(t)^{nq'+q}.
\end{equation*}
So the volume non-collapsing estimate in (\ref{volume of geodesic balls}) reads
\begin{equation}\label{almost-sharp-volume-estimate}
\frac{\vol(B_r(x_0))}{[\omega^n]}\ge \frac{C\cdot r^{2n}}{g_1\left(\frac{1}{r}\right)^{2n}\cdots g_\ell\left(\frac{1}{r}\right)^{2n}g_{\ell+1}\left(\frac{1}{r}\right)^{nq'+q}},
\end{equation}
where $C=C(\omega_0,\Phi,\Phi_1,E; A,K,L,K,n)$. This is exactly the volume estimate proved in \cite{ZwZy25+} (see also \cite{GuPhSt24,GuPhSoSt23,GuPhSoSt24-1,Liu24+,GuTo25,Vu24+}).






\section{Examples}\label{example section}

In the section we consider several typical examples of $N$-functions. All the functions are smooth and strictly convex.

Let $\Phi$ be an $N$-function and $\Phi^*$ be the complementary. Assume that $\Phi'(t)$ increases strictly in $t$. By definition, $\Phi^*(s)=\sup\{st-\Phi(t)\,|\,t\ge 0\}$. For any $s>0$, the supremum of $st-\Phi(t)$ achieve at a unique point $t^*$, then
\begin{equation}
\Phi'(t^*)=s.
\end{equation}
It follows that
\begin{equation}
\Phi^*(s)=t^*\cdot\Phi'(t^*)-\Phi(t^*).
\end{equation}
One then solves out $t^*=t^*(s)$ from the first identity and substitutes into the second one to get the required function $\Phi^*(s)$.

In the following, we write $f(s)\prec g(s)$ for two functions on $[0,\infty)$ if there exist constants $C_1,C_2>0$ such that
$$f(s)\le C_1\cdot g\left(C_2 t\right),$$
whenever $t$ is sufficiently large. If $f\prec g$ and $g\prec f$, then we write $f\sim g$.

\subsection{Exponential growth function}

\begin{example}\label{Exponential growth function}
Let $a>0$ be a constant. The convex conjugate of
$$\Phi(t)=e^{at}-1-at$$
is
$$\Phi^*(s)=\left(1+\frac{s}{a}\right)\log\left(1+\frac{s}{a}\right)-\frac{s}{a}.$$
Thus, the function $h=h(t)$ is the unique solution to
$$\left(1+\frac{h}{a}\right)\log\left(1+\frac{h}{a}\right)-\frac{h}{a}=e^t.$$
When $t\rightarrow\infty$, the function has an asymptotic
$$h(t)\sim\frac{a}{t}\cdot e^t.$$
The function on the right hand side is convex for large $t$.

Let $\Phi(t)=e^{at}-1-at$ for some $a>0$. Then, $h(t)\ge\tilde{h}(t)=\frac{a}{2t}e^t$ for large $t$. So,
$$(h^*)^{-1}(t)\ge(\tilde{h}^*)^{-1}(t)\sim\frac{t}{\log t}$$
for large $t$. Then, the integration term in (\ref{modulus estimate of CMA: 00}) can be estimated as follows,
\begin{eqnarray*}
\int_{\tau}^\infty\frac{dt}{t^{1/n}\cdot(h^*)^{-1}(t)}\le C\int_{t_0}^\infty\frac{\log t dt}{t^{1+\frac{1}{n}}}
\le C\cdot t_0^{-\frac{1}{n}}\cdot\log t_0.
\end{eqnarray*}
Then (\ref{modulus estimate of CMA: 00}) gives
\begin{eqnarray*}
\psi_\delta\le\varphi_0+\delta\cdot\Phi^*(t_0)+C\cdot t_0^{-\frac{1}{n}}\cdot\log t_0.
\end{eqnarray*}
Choosing $t_0=\delta^{-\frac{n}{1+n}}$ one gets
\begin{equation}\label{stability example: 0}
\psi_\delta-\varphi_0\le C\cdot \hbar(\delta),
\end{equation}
where $$\hbar(\delta)=\delta^{\frac{1}{1+n}}\cdot\left(-\log\delta\right).$$
Notice that the power of $\delta$ on the right hand side is independent of $a>0$.
\end{example}

\subsection{Polynomial growth function}

\begin{example}\label{Example-poly}

If $\Phi(t)=\frac{t^p}{p}$ for some $p> 1$, then
$$\Phi^*(s)=\frac{s^{q}}{q},\hspace{0.5cm}h(t)=q^{\frac{1}{q}}\cdot e^{\frac{t}{q}},$$
where $q$ is the dual index of $p$ in the sense that $\frac{1}{p}+\frac{1}{q}=1$.

Obviously $h$ is convex in $t$. Then, the integration term in (\ref{modulus estimate of CMA: 00}) can be estimated
\begin{eqnarray*}
\int_{t_0}^\infty\frac{dt}{t^{1/n}\cdot(h^*)^{-1}(t)}
\le\int_{t_0}^\infty\frac{h^{-1}(t)dt}{t^{1+\frac{1}{n}}}
\le C\cdot t_0^{-\frac{1}{n}}\cdot\log t_0.
\end{eqnarray*}
Then (\ref{modulus estimate of CMA: 00}) gives
\begin{eqnarray*}
\psi_\delta\le\varphi_0+\delta\cdot\Phi^*(t_0)+C\cdot t_0^{-\frac{1}{n}}\cdot\log t_0.
\end{eqnarray*}
Choosing $t_0=\delta^{-\frac{n}{1+nq}}\cdot(-\log\delta)^{\frac{n}{1+nq}}$ one gets
\begin{equation}\label{stability example: 1}
\psi_\delta\le\varphi_0+C\cdot\hbar(\delta),
\end{equation}
where $$\hbar(\delta)=\delta^{\frac{1}{1+nq}}\cdot\left(-\log\delta\right)^{\frac{nq}{1+nq}}.$$ The estimate is slightly better than \cite[Theorem 2.1]{Ko08}, \cite[Proposition 2.9]{DDGHKZ}. See also \cite[Theorem 3.4]{EyGuZe09}, \cite[Proposition 5.2]{GuZe12} for some related estimate.
\end{example}

\subsection{Slow growth function}\label{slow example}

Define a sequence of functions
$$g_0(t)=t,\,g_1(t)=\log (1+t),\,\cdots,g_{k+1}(t)=\log(1+g_k(t)),\cdots.$$
Given $k\ge 1$ and a sequence of real numbers $p_0,p_1,\cdots,p_{k}$, define
$$\Phi(t)=g_0(t)^{p_0}\cdot g_1(t)^{p_1}\cdots g_k(t)^{p_k}.$$
If $p_0\ge 1$ and the first non-vanishing $p_i$ with $i\ge 1$ is positive, then $\lim_{t\rightarrow\infty}\Phi(t)/t=\infty$, so its complementary $\Phi^*$ is well-defined. If each $p_i\ge 1$, then $\Phi$ is convex, which can be checked by induction.

Let us calculate the complementary $\Phi^*$ if it exists. For any $s>0$, let $t^*$ be the time so that $st^*-\Phi(t^*)=\Phi^*(s)$. Then, by the definition of $\Phi^*$, at $t^*$,
\begin{equation}
s=\Phi'(t^*)=\left(\frac{p_0}{t^*}+\sum_{j=1}^k\frac{p_j}{(1+g_0)\cdots(1+g_{j-1})\cdot g_j}\right)\cdot\Phi(t^*).
\end{equation}
It follows that
\begin{eqnarray}
\Phi^*(s)=\left(p_0-1+\sum_{j=1}^k\frac{p_j\cdot g_0}{(1+g_0)\cdots(1+g_{j-1})\cdot g_j}\right)\cdot\Phi(t^*),
\end{eqnarray}
where $g_j=g_j(t^*)$.

\begin{example}[Case $p_0>1$]
Assume $p_0>1$. When $s\rightarrow\infty$, we have the asymptotics
$$s\sim\frac{p_0}{t^*}\cdot\Phi(t^*)= p_0\cdot g_0^{p_0-1}\cdot g_1^{p_1}\cdots g_k^{p_k},$$
so,
$$t^*\sim \left(\frac{s}{p_0}\right)^{\frac{1}{p_0-1}}\cdot\tilde{g}_1^{-\frac{p_1}{p_0-1}}\cdots\tilde{g}_k^{-\frac{p_k}{p_0-1}}$$
where $\tilde{g}_j=g_j\left(\left(\frac{s}{p_0}\right)^{\frac{1}{p_0-1}}\right)$. Obviously, $\tilde{g}_1\sim \frac{1}{p_0-1}g_1(s)$ and $\tilde{g}_j\sim g_j(s)$ for $j\ge 2$. It follows
\begin{eqnarray}
\Phi^*(s)\sim\Phi(t^*)
\sim s^{\frac{p_0}{p_0-1}}\cdot g_1(s)^{-\frac{p_1}{p_0-1}}\cdots g_k(s)^{-\frac{p_k}{p_0-1}},
\end{eqnarray}
as $s\rightarrow\infty$. It reduces to the previous example when $p_j=0$ for $j\ge 1$.
\end{example}

\begin{example}[Case $p_0=1$]\label{Example-log-p=1}
For any $s>0$, let $t^*$ be the time so that $st^*-\Phi(t^*)=\Phi^*(s)$. Then,
\begin{eqnarray*}
\Phi^*(s)=\left(\sum_{j=1}^k\frac{p_j\cdot g_0}{(1+g_0)\cdots(1+g_{j-1})\cdot g_j}\right)\cdot\Phi(t^*)\sim \frac{\Phi(t^*)}{g_1},
\end{eqnarray*}
where $g_j=g_j(t^*)$. When $s\rightarrow\infty$, we have the asymptotics
$$s\sim\frac{1}{t^*}\cdot\Phi(t^*)=g_1^{p_1}\cdots g_k^{p_k}.$$
Similar calculation gives
$$g_1(t^*)\sim s^{\frac{1}{p_1}}\cdot\hat{g}_1(s)^{-\frac{p_2}{p_1}}\cdots\hat{g}_{k-1}(s)^{-\frac{p_k}{p_1}}$$
and
\begin{eqnarray*}
\Phi^*(s)\sim p_1\cdot s^{1-\frac{1}{p_1}}\cdot\hat{g}_1(s)^{\frac{p_2}{p_1}}\cdots\hat{g}_{k-1}(s)^{\frac{p_k}{p_1}}
\cdot\exp\left\{s^{\frac{1}{p_1}}\cdot\hat{g}_1(s)^{-\frac{p_2}{p_1}}\cdots\hat{g}_{k-1}(s)^{-\frac{p_k}{p_1}}\right\}
\end{eqnarray*}
where $\hat{g}_j(s)=g_j\left(s^{\frac{1}{p_1}}\right)$. Obviously, $\hat{g}_1\sim\frac{1}{p_1}g_1$ and $\hat{g}_j\sim g_j$ for $j\ge 2$. It follows
\begin{eqnarray}\label{Phi68-asymptotic}
\Phi^*(s)\sim p_1^{1-\frac{p_2}{p_1}}\cdot s^{1-\frac{1}{p_1}}\cdot g_1(s)^{\frac{p_2}{p_1}}\cdots g_{k-1}(s)^{\frac{p_k}{p_1}}
\cdot\exp\left\{p_1^{\frac{p_2}{p_1}}\cdot s^{\frac{1}{p_1}}\cdot g_1(s)^{-\frac{p_2}{p_1}}\cdots g_{k-1}(s)^{-\frac{p_k}{p_1}}\right\}.
\end{eqnarray}
\end{example}

\begin{example}[Case $p_0=1$: the function $h$]
Suppose $h(t)=u$, then
$$t=\log\Phi^*(u)\sim p_1^{\frac{p_2}{p_1}}\cdot u^{\frac{1}{p_1}}\cdot g_1(u)^{-\frac{p_2}{p_1}}\cdots g_{k-1}(u)^{-\frac{p_k}{p_1}}.$$
It can be solved out that
\begin{eqnarray}\label{h(t)=Phi(t)/t}
h(t)\sim p_1^{-p_2}\cdot t^{p_1}\cdot g_1(v)^{p_2}\cdots g_{k-1}(v)^{p_k}
\sim t^{p_1}\cdot g_1(t)^{p_2}\cdots g_{k-1}^{p_k}.
\end{eqnarray}
Here we used a variable $v=p_1^{-p_2}\cdot t^{p_1}$. So $h$ is asymptotic to a convex function. From above, we see that
\begin{equation}\label{Phi(t)/t=h(logt)}
h(\log t)\sim \frac{\Phi(t)}{t}.
\end{equation}
\end{example}

\begin{example}[Case $p_0=1$: the function $\hbar$]\label{Example-slow}
Suppose that $\Phi$ has a special form
$$\Phi(t)=g_0(t)\cdot g_1(t)^n\cdot\cdots\cdot g_{k-1}(t)^n\cdot g_k(t)^p,$$
where $p>n$. When $k\ge 2$, by (\ref{Phi(t)/t=h(logt)}) we have
\begin{equation}\label{Example-slow-h(t)}
h(t)\sim \tilde{h}(t)=t^n\cdot g_1(t)^n\cdots g_{k-1}(t)^p
\end{equation}
whenever $t$ is sufficiently large. It is obvious that $\tilde{h}$ is an $N$-function, so
$$(h^*)^{-1}(t)\succ (\tilde{h}^*)^{-1}(t)\ge\frac{t}{\tilde{h}^{-1}(t)}.$$
So, for $t_0>>1$, the integration
\begin{eqnarray*}
\int_{t_0}^\infty\frac{dt}{t^{1/n}\cdot(h^*)^{-1}(t)}&\le&\int_{t_0}^\infty\frac{\tilde{h}^{-1}(t)dt}{t^{1+\frac{1}{n}}}\\
&\le& C\cdot\int_{t_0}^\infty\frac{dt}{t\cdot g_1(t)\cdots\cdot g_{k-1}(t)^{\frac{p}{n}}}\\
&\le& C\cdot g_{k-1}(t_0)^{1-\frac{p}{n}}.
\end{eqnarray*}
Then the estimate (\ref{modulus estimate of CMA: 00}) reads $\psi_\delta\le\varphi_0+\delta\cdot\Phi^*(t_0)+C\cdot g_{k-1}(t_0)^{1-\frac{p}{n}}$. Taking infimum of $\delta\cdot\Phi^*(t_0)+C\cdot g_{k-1}(t_0)^{1-\frac{p}{n}}$ over $t_0$ one gets
\begin{equation}\label{stability example: 2}
\psi_\delta\le\varphi_0+C\cdot g_{k-1}(-\log\delta)^{-\frac{p-n}{n}}.
\end{equation}
When $k=1$ we have $h(t)\succ t^p$. The integration term
\begin{eqnarray*}
\int_{t_0}^\infty\frac{dt}{t^{1/n}\cdot(h^*)^{-1}(t)}\le C\cdot\int_{t_0}^\infty\frac{dt}{t^{1+\frac{1}{n}-\frac{1}{p}}}
\le C\cdot t_0^{-\frac{p-n}{np}}.
\end{eqnarray*}
Then the estimate (\ref{modulus estimate of CMA: 00}) reads $\psi_\delta\le\varphi_0+\delta\cdot\Phi^*(t_0)+C\cdot t_0^{-\frac{p-n}{np}}$, where $\Phi^*(t_0)\sim p\cdot t_0^{1-\frac{1}{p}}\cdot\exp\{t_0^{\frac{1}{p}}\}$. Choosing $$t_0=\left(-\log\delta+\frac{np-2n+p}{n}\log(-\log\delta)\right)^p$$
one gets
\begin{equation}\label{stability example: 3}
\psi_\delta\le\varphi_0+C\cdot (-\log\delta)^{-\frac{p-n}{n}},
\end{equation}
The conclusion coincides with (\ref{stability example: 2}) when $k=1$. It is the case that was considered by Guo-Phong-Tong-Wang \cite{GuPhToWa21}.
\end{example}



\end{document}